\documentclass[10pt]{amsart}
\usepackage{latexsym,amsmath,amssymb,graphicx}
\usepackage[all,web]{xy}
\usepackage{color}

\def\urlfont{\DeclareFontFamily{OT1}{cmtt}{\hyphenchar\font='057}
              \normalfont\ttfamily \hyphenpenalty=10000}

\DeclareFontFamily{OT1}{rsfs10}{}
\DeclareFontShape{OT1}{rsfs10}{m}{n}{ <-> rsfs10 }{}
\DeclareMathAlphabet{\mathscript}{OT1}{rsfs10}{m}{n}

\DeclareMathOperator{\im}{Im}       
\DeclareMathOperator{\Spec}{Spec}   
\DeclareMathOperator{\Proj}{Proj}   
\DeclareMathOperator{\Hom}{Hom}     
\DeclareMathOperator{\Pic}{Pic}     
\DeclareMathOperator{\rk}{rk}       
\DeclareMathOperator{\diag}{diag}   
\DeclareMathOperator{\fan}{fan}     
\DeclareMathOperator{\conv}{Conv}   
\DeclareMathOperator{\lcm}{lcm}     

\title[Weighted Projective Spaces from the toric point of view]{Weighted Projective Spaces from the toric point of view with computational applications.}

\author[M. Rossi and L.Terracini]{Michele Rossi and Lea Terracini}

\dedicatory{Dedicated to our teacher, colleague and friend Luciana Picco Botta}

\address{Dipartimento di Matematica, Universit\`a di Torino,
via Carlo Alberto 10, 10123 Torino} \email{michele.rossi@unito.it,
lea.terracini@unito.it}

\thanks{The authors were partially supported by the Local Project \lq\lq Computational Algebra and Applications\rq\rq (2007) and the MIUR-PRIN 2008 Research Funds}

\def \a{\alpha }
\def \b{\beta }
\def \d{\delta }

\def \s{\sigma }
\def \D{\Delta }

\def \Si{\Sigma }

\def \e{\mathbf{e}}
\def \u{\mathbf{u}}
\def \v{\mathbf{v}}
\def \n{\mathbf{n}}
\def \w{\mathbf{w}}
\def \t{\mathbf{t}}

\def\P{{\mathbb{P}}}

\def\p2{\mathbb{P}^2}
\def\p3{\mathbb{P}^3}
\def\p4{\mathbb{P}^4}

\def\rk{\operatorname{rk}}

\def\GL{\operatorname{GL}}

\def\Mat{\operatorname{Mat}}

\def\Z{\mathbb{Z}}

\def\C{\mathbb{C}}
\def\R{\mathbb{R}}

\def\Q{\mathbb{Q}}
\def\N{\mathbb{N}}

\def\W{\mathbb{W}}

\theoremstyle{plain}
\newtheorem{theorem}{Theorem}[section]
\newtheorem{proposition}[theorem]{Proposition}
\newtheorem{criterion}[theorem]{Criterion}
\newtheorem{thm-def}[theorem]{Theorem--Definition}
\newtheorem{corollary}[theorem]{Corollary}

\newtheorem{lemma}[theorem]{Lemma}

\newtheorem*{a-proposition}{Proposition}

\theoremstyle{remark}
\newtheorem{remark}[theorem]{Remark}

\newtheorem{example}[theorem]{Example}

\newtheorem*{notation}{Notation}

\theoremstyle{definition}
\newtheorem{definition}[theorem]{Definition}
\newtheorem{algorithm}[theorem]{Algorithm}

\newtheorem*{step I}{Step I}
\newtheorem*{step II}{Step II}
\newtheorem*{step III}{Step III}
\newtheorem*{step IV}{Step IV}

\newtheorem*{acknowledgements}{Acknowledgements}

\newcommand{\oneline}{\vskip12pt}
\newcommand{\halfline}{\vskip6pt}
\newcommand{\cy}{Ca\-la\-bi--Yau }

\begin{document}


\begin{abstract} The purpose of the present paper is threefold. First: giving a treatise on weighted projective spaces by the toric point of view. Second: providing characterizations of fans and polytopes giving weighted projective spaces, with particular focus on a kind of \emph{recognition process} of toric data like fans and polytopes. Third: building a mathematical framework for the algorithmic and computational approach to wps's realized in \cite{RTpdf,RT-Maple}.
\end{abstract}

\maketitle

\tableofcontents

\section*{Introduction}

The purpose of the present paper is threefold.

 \halfline The first aim is to propose a treatise on weighted projective spaces (wps's) from the toric point of view. In fact there is a quite vast literature on the subject but this is rather fragmented and scattered in a variety of places due to the multitude of applications that wps's have had over the years. Therefore we would like to propose a, as much as possible, self-contained toric approach to wps's, giving characterizations of their toric data, like fans and polytopes, and producing algorithms that allow them to be easily determined, even by hand. Actually, a wps is one of the simplest examples of a complete normal toric variety, probably meaning that a toric treatise on this subject could be regarded as not much more than a sort of instructive exercise on toric geometry. Nevertheless wps's, and toric varieties in general, are often considered in many applications, like e.g. super--string theory to produce examples of Calabi--Yau threefolds modeling of the space time of hidden dimensions, whose scholars are not expected to have a sufficiently deep background in algebraic geometry. Since we are convinced that combinatorics and convex geometry require a smaller background than algebraic geometry, we believe that a completely toric treatise of the issue may provide a good service to many people. However, such a treatise should not be thought of as a comprehensive survey on the subject: wps's have so many steppingstones with so many mathematical areas and so many applications that giving a comprehensive survey would probably mean writing a whole book on the subject. By the way, we also believe that giving a comparison between well--known algebraic geometric results and their toric counterparts, like e.g. the \emph{Reduction Theorem} \ref{thm:riduzione} asserting the isomorphism $\P(Q)\cong\P(Q')$ when $Q'$ is the reduced weights vector of $Q$ (see \ref{ssez:divisori}), may be a mathematical outcome worthy of some interest.

 \halfline The second aim is to provide, among others, particular characterizations of toric data (fans and polytopes) defining a wps, which, as far as we know, seem to be still unpublished. Namely:
 \begin{enumerate}
   \item Proposition \ref{prop:HNFdue} providing a surprising link between a wps fan and the switching matrix computing the Hermite Normal Form $B=(1,0,\ldots,0)^T$ of the transposed weights vector $Q^T=(q_0,\ldots,q_n)^T$,
   \item Definition \ref{def:Wtrasversa} and the following Theorem \ref{thm:fan-politopi} proposing the \emph{weighted trans\-ver\-sion} process as an easy and quick method to produce a \emph{minimally polarizing polytope} of a wps starting from its fan,
   \item Theorems \ref{thm:(A)} and \ref{thm:(B)} giving characterizations of polytopes defining a polarized wps.
 \end{enumerate}
 Let us discuss these results in more detail. Proposition \ref{prop:HNFdue} asserts that a fan of $\P(Q)$ is encoded in a matrix $U$ such that $U\cdot Q^T=B$, whose existence is guaranteed by Hermite Normal Form (HNF) algorithm, building a bridge between toric geometry and linear algebra. This is proved by a direct check of equivalent conditions in Theorem \ref{thm:fan}. As a consequence a procedure computing a wps fan can be implemented in any mathematical software treating basic linear algebra procedures like the one giving the HNF of a matrix (see algorithm \ref{alg:daQaF}).

 \noindent Results listed in (2) and (3) are naturally obtained by constructing the polytope of a minimally polarizing divisor of $\P(Q)$ starting from its fan: when restricted to the associated matrices, such a process turns out to be, up to the multiplication by a diagonal matrix of weights, nothing more then taking the \emph{trans}posed in\emph{verse} (hence ``\emph{transverse}'') matrix of the fan matrix after the deletion of the column corresponding to a chosen weight $q_0$. Then, conditions stated in Theorems \ref{thm:(A)} and \ref{thm:(B)} allowing the "recognition" of the polytope defining the polarized wps $(\P(Q),\mathcal{O}(m))$ are obtained by checking the equivalent conditions in Theorem \ref{thm:fan} after having inverted the transverse process (see Proposition \ref{prop:admuno} and the following Remark \ref{rem:inversione}).

 \halfline Last, but not least, the third aim of this work is to provide a mathematical framework for an algorithmic approach to wps's. This is part of a much bigger project, aiming to produce a large number of procedures for an ``automatic'' treatment of toric varieties, which we are carrying on in collaboration with our student Massimiliano Povero \cite{Povero}, \cite{PRT}. The main motivation for such a computational approach  is that toric varieties are a very fertile ground for the production of examples in algebraic geometry and for many applications. Let us say that these last years have seen a proliferation of mathematical packages on these topics which we do not quote here to avoid forgetting someone. We just refer the interested reader to the well updated D.~Cox web--page \cite{Cox-wp}. However, this is a clear sign of a significant interest in the topic.

\noindent \emph{The recognition processes for fans and polytopes} of a given wps, giving rise from results collected in the present paper, seem to lead to a way of approaching the study of toric varieties which, at least in principle, may be useful for many applications where the geometry introduced by the recognition of toric data can help to understand some involved problem (see e.g. \cite{DSS}). At this purpose let us say that the recognition process for toric data of a wps, here given, can be reasonably extended to some more general toric varieties, like e.g. finite quotients of wps, so called \emph{fake wps} \cite{Kasprzyk}, and products of them \cite{RT2}.

 \noindent In the present paper we have summarized the computational applications of the above mentioned results with the following four algorithms. For the complete list of computational procedures we refer the interested reader to \cite{RTpdf} and to their Maple implementation \cite{RT-Maple}.
 \begin{itemize}
   \item Algorithm \ref{alg:fan} recognizing a fan of $\P(Q)$, as an application of Theorem \ref{thm:fan}.
   \item Algorithm \ref{alg:daQaF} providing a fan of $\P(Q)$, as an application of Proposition \ref{prop:HNFdue}.
   \item Algorithm \ref{alg:daQaP} providing a polytope of the minimally polarized $(\P(Q),\mathcal{O}(1))$, as an application of Theorem \ref{thm:fan-politopi}.
   \item Algorithm \ref{alg:politopo} recognizing a polytope giving $(\P(Q),\mathcal{O}(m))$, as an application of Theorems \ref{thm:(A)} and \ref{thm:(B)}.
 \end{itemize}

 \begin{acknowledgements} A very first draft of this paper was prepared during a series of lectures on toric varieties held at the Department of Mathematics of the University of Turin on February 2003. The authors would like to thank all the participants and contributors in that series of lectures and in particular A. Albano, A. Collino, A. Grassi, L. Picco Botta, M. Roggero. The authors would also like to thank V. Perduca, a student of A. Grassi who used such a first draft in his PHD thesis, for stimulating questions and conversations. In particular they are indebted to A. Grassi for her encouragement to complete that draft and write down the present paper. A heartfelt thanks goes to C. Casagrande who first pointed us the reference \cite{Conrads} and to S. di Rocco for suggestions and stimulating conversations during the School and Workshop on ``Tropical and Toric Geometry", held in Trento on September 2011: special thanks go to all the participants in that conference and in particular to the organizers, G. Casnati, C. Fontanari, F. Galluzzi, R. Notari, F. Vaccarino and, last, but not least, the C.I.R.M. of Trento and A. Micheletti, for the warm and stimulating atmosphere they were able to create. Finally the authors would like to thank O.Fujino for useful remarks and suggestions.
 \end{acknowledgements}

\section{Preliminaries and notation}

\subsection{Toric varieties}

A \emph{$n$--dimensional toric variety} is an algebraic normal variety $X$ containing the \emph{torus} $T:=(\C^*)^n$ as a Zariski open subset such that the natural multiplicative self--action of the torus can be extended to an action $T\times X\rightarrow X$.

Let us quickly recall the classical approach to toric varieties by means of \emph{cones} and \emph{fans}. For proofs and details the interested reader is referred to the extensive treatments \cite{Danilov}, \cite{Fulton}, \cite{Oda} and the recent and quite comprehensive \cite{CLS}.

\noindent As usual $M$ denotes the \emph{group of characters} $\chi : T \to \C^*$ of $T$ and $N$ the \emph{group of 1--parameter subgroups} $\lambda : \C^* \to T$. It follows that $M$ and $N$ are $n$--dimensional dual lattices via the pairing
\begin{equation*}
\begin{array}{ccc}
M\times N & \longrightarrow & \Hom(\C^*,\C^*)\cong\C^*\\
 \left( \chi,\lambda \right) & \longmapsto
& \chi\circ\lambda
\end{array}
\end{equation*}
which translates into the standard paring $\langle u,v\rangle=\sum u_i v_i$ under the identifications $M\cong\Z^n\cong N$ obtained by setting $\chi(\t)=\t^{\u}:=\prod t_i^{u_i}$ and $\lambda(t)=t^{\v}:=(t^{v_1},\ldots,t^{v_n})$.

\subsubsection{Cones and affine toric varieties}

Define $N_{\R}:=N\otimes \R$ and $M_{\R}:=M\otimes\R\cong \Hom(N,\Z)\otimes\R \cong \Hom(N_{\R},\R)$.

\noindent A \emph{convex polyhedral cone} (or simply a \emph{cone}) $\sigma$ is the subset of $N_{\R}$ defined by
\begin{equation*}
    \sigma = \langle \v_1,\ldots,\v_s\rangle:=\{ r_1 \v_1 + \dots + r_s \v_s \in N_{\R} \mid r_i\in\R_{\geq 0} \}
\end{equation*}
The $s$ vectors $\v_1,\ldots,\v_s\in N_{\R}$ are said \emph{to generate} $\sigma$. A cone $\s=\langle \v_1,\ldots,\v_s\rangle$ is called \emph{rational} if $\v_1,\ldots,\v_s\in N$, \emph{simplicial} if $\v_1,\ldots,\v_s$ are $\R$--linear independent and \emph{non-singular} if $\v_1,\ldots,\v_s$ can be extended by $n-s$ further elements of $N$ to give a basis of the lattice $N$.

\noindent A cone $\s$ is called \emph{strictly convex} if it does not contain a linear subspace of $N_{\R}$.

\noindent The \emph{dual cone $\s^{\vee}$ of $\s$} is the subset of $M_{\R}$ defined by
\begin{equation*}
    \sigma^{\vee} = \{ \u \in M_{\R} \mid \forall\ \v \in \sigma \quad \langle \u, \v \rangle \ge 0 \}
\end{equation*}
A \emph{face $\tau$ of $\s$} (denoted by $\tau <\s$) is the subset defined by
\begin{equation*}
    \tau = \sigma \cap \u^{\bot} = \{\v \in \sigma \mid \langle \u, \v \rangle = 0 \}
\end{equation*}
for some $\u\in \sigma ^{\vee}$. Observe that also $\tau$ is a cone.

\noindent Gordon's Lemma (see \cite{Fulton} \S 1.2, Proposition 1) ensures that the semigroup $S_{\s}:=\s^{\vee}\cap M$ is \emph{finitely generated}. Then also the associated $\C$--algebra $A_{\s}:=\C[S_{\s}]$ is finitely generated. A choice of $r$ generators gives a presentation of $A_{\s}$
\begin{equation*}
    A_{\s}\cong \C[X_1,\dots,X_r]/I_{\s}
\end{equation*}
where $I_{\s}$ is the ideal generated by the relations between generators. Then
\begin{equation*}
    U_{\s}:=\mathcal{V}(I_{\s})\subset\C^r
\end{equation*}
turns out to be an \emph{affine toric variety}. In other terms an affine toric variety is given by $U_{\s}:=\Spec(A_{\s})$. Since a closed point $x\in U_{\s}$ is an evaluation of elements in $\C[S_{\s}]$ satisfying the relations generating $I_{\s}$, then it can be identified with a semigroup morphism $x:S_{\s}\rightarrow\C$ assigned by thinking of $\C$ as a multiplicative semigroup. In particular the \emph{characteristic morphism}
\begin{equation}\label{caratteristico}
\begin{array}{cccc}
x_{\s}&:\s^{\vee}\cap M & \longrightarrow & \C\\
      &      \u & \longmapsto & \left\{\begin{array}{cc}
                                         1 & \text{if $\u\in\s^{\bot}$} \\
                                         0 & \text{otherwise}
                                       \end{array}
      \right.
\end{array}
\end{equation}
which is well defined since $\s^{\bot}<\s^{\vee}$, defines a \emph{characteristic point} $x_{\s}\in U_{\s}$ whose toric orbit $O_{\s}$ turns out to be a $(n-\dim(\s))$--dimensional torus embedded in $U_{\s}$ (see e.g. \cite{Fulton} \S 3).

\subsubsection{Fans and toric varieties}

A \emph{fan} $\Si$ is a finite set of cones $\s\subset N_{\R}$ such that
\begin{enumerate}
  \item for any cone $\s\in\Si$ and for any face $\tau<\s$ then $\tau\in\Si$,
  \item for any $\s,\tau\in\Si$ then $\s\cap\tau<\s$ and $\s\cap\tau<\tau$.
\end{enumerate}
For any $i$ with $0\leq i\leq n$ denote by $\Si(i)\subset \Si$ the subset of $i$--dimensional cones, called the \emph{$i$--skeleton of $\Si$}. A fan $\Si$ is called \emph{simplicial} if any cone $\s\in\Si$ is simplicial and \emph{non-singular} if any such cone is non-singular. The \emph{support} of a fan $\Si$ is the subset $|\Si|\subset N_{\R}$ obtained as the union of all of its cones i.e.
\begin{equation*}
    |\Si|:= \bigcup_{\s\in\Si} \s \subset N_{\R}\ .
\end{equation*}
If $|\Si|=N_{\R}$ then $\Si$ will be called \emph{complete} or \emph{compact}.

Since for any face $\tau <\s$ the semigroup $S_{\s}$ turns out to be a sub-semigroup of $S_{\tau}$, there is an induced immersion $U_{\tau}\hookrightarrow U_{\s}$ between the associated affine toric varieties which embeds $U_{\tau}$ as a principal open subset of $U_{\s}$. Given a fan $\Si$ one can construct \emph{an associated toric variety $X(\Si)$} by patching all the affine toric varieties $\{U_{\s}\ |\ \s\in\Si \}$ along the principal open subsets associated with any common face. Moreover \emph{for every toric variety $X$ there exists a fan $\Si$ such that $X\cong X(\Si)$} (see \cite{Oda} Theorem 1.5). It turns out that (\cite{Oda} Theorems 1.10 and 1.11; \cite{Fulton} \S 2):
\begin{itemize}
  \item \emph{$X(\Si)$ is non-singular if and only if the fan $\Si$ is non-singular,}
  \item \emph{$X(\Si)$ is complete if and only if the fan $\Si$ is complete.}
\end{itemize}

 In the following a \emph{$1$--generated fan} $\Si$ is a \emph{fan generated by a set of $n+1$ integral vectors} i.e. a fan whose cones $\s\subset N\otimes\R$ are generated by any proper subset of a given finite subset $\{\v_0,\ldots,\v_n\}\subset N$: we will write
\begin{equation}\label{def:genfan}
    \Si=\fan(\mathbf{v}_0,\ldots,\mathbf{v}_n)\ .
\end{equation}
Given a $1$--generated fan $\Si=\fan(\v_0,\ldots,\v_n)$, the matrix $V=(\v_0,\ldots,\v_n)$ will be called \emph{a fan matrix of $\Si$}. Notice that $\Si$ determines $V$ up to a permutations of columns, meaning that $\Si$ admits $(n+1)!$ associated fan matrices.

\noindent If $V=(\v_0,\ldots,\v_n)$ is a fan matrix of $\Si=\fan(\v_0,\ldots,\v_n)$ then we will denote the maximal square sub-matrices of $V$ and the associated $n$--minors as follows
\begin{equation}\label{Vj}
    \forall\ 0\leq j\leq n\quad V^j:=(\v_0,\ldots,\v_{j-1},\v_{j+1},\ldots,\v_n)\ ,\ V_j=\det(V^j)\ .
\end{equation}

\subsubsection{Polytopes and projective toric varieties}\label{sssez:politopi}

A \emph{polytope} $\D\subset M_{\R}$ is the convex hull of a finite set of points. If this set is a subset of $M$ then the polytope is called \emph{integral}. Starting from an integral polytope one can construct a projective toric variety as follows. Here we will follow the approach of \cite{Batyrev}, which the interested reader is referred to for proofs and details (see also \cite{CK} \S 3.2.2).

\noindent For any $k\in\N$ one can define the dilated polytope $k\D:=\{k\u\ |\ \u\in\D\}$. It is then possible to define a graded $\C$--algebra $S_{\D}$, associated with the integral polytope $\D$, as follows. For any $\u\in k\D\cap M$ consider the associated character $\chi^{\u}:\t\mapsto\t^{\u}$. Given $t\in\C^*$ consider the \emph{monomial} $t^k\chi^\u:\t\mapsto t^k\t^{\u}$. It well defines a \emph{monomial product} $t^{k_1}\chi^{\u_1}\cdot t^{k_2}\chi^{\u_2}:=t^{k_1+k_2}\chi^{\u_1+\u_2}$ where $\u_1+\u_2\in (k_1+k_2)\D$. Let $S_{\D}$ be the $\C$--algebra generated by all monomials $\{t^k\chi^{\u}\ |\ k\in\N\ ,\ \u\in k\D\}$ which is a graded object by setting $\deg(t^k\u)=k$.

\noindent The \emph{projective} variety $\P_{\D}:=\Proj(S_{\D})$ turns out to be naturally a \emph{toric variety} whose fan $\Si_{\D}$ can be recovered as follows. For any non empty face $F<\D$ consider the cone
\begin{equation*}
    \check{\s}_F:=\{r(\u-\u')\ |\ \u\in\D\ ,\ \u'\in F\ ,\ r\in\R_{\geq 0}\}\subset M_{\R}
\end{equation*}
and define $\s_F:=\check{\s}^{\vee}_F\subset N_{\R}$. Then $\Si_{\D}:=\{\s_F\ |\ F<\D\}$ turns out to be a fan, called the \emph{normal fan} of the polytope $\D$, such that there exists a very ample divisor $H$ of $X(\Si_{\D})$ for which $(X(\Si_{\D}),H)\cong(\P_{\D},\mathcal{O}(1)$, where $\mathcal{O}(1)$ is the natural polarization of $\P_{\D}=\Proj(S_{\D})$ (see \cite{Batyrev} Proposition 1.1.2).

Viceversa a \emph{projective toric variety} is the couple $(X(\Si),H)$ of a toric variety $X(\Si)$ and a polarization given by (the linear equivalence class of) a hyperplane section $H$. For any 1-cone $\rho\in\Si(1)$, consider the toric stable divisor $D_{\rho}:=\overline{O}_{\rho}$ defined as the closure of the toric orbit of the characteristic point $x_{\rho}$, defined in (\ref{caratteristico}). Since those divisors generate the Chow group of Weil divisors $A_{n-1}(X(\Si))$ (see \cite{Fulton} \S 3.4), there exist complex coefficients $a_{\rho}\in\C$ such that $H=\sum_{\rho\in\Si(1)}a_{\rho} D_{\rho}$. It is then well defined the integral polytope
\begin{equation}\label{politopo}
    \D_H:=\{\u\in M_{\R}\ |\ \forall \rho\in\Si(1)\ \langle\u,\n_{\rho}\rangle\geq - a_{\rho}\}
\end{equation}
where $\n_{\rho}$ is the unique generator of the semigroup $\rho\cap N$. Then
\begin{equation*}
    (\P_{\D_H},\mathcal{O}(1))\cong (X(\Si),H)\ .
\end{equation*}

\subsubsection{Divisors}\label{sssez:divisori} Since a toric variety $X=X(\Si)$ is normal (\cite{Danilov} Proposition 3.2) hence it makes sense to talk about its Weil divisors, whose group will be denoted by $\mathcal{W}(X)$. Let $\mathcal{C}(X)\subset\mathcal{W}(X)$ denote the subgroup of Cartier divisors. The subgroup of \emph{toric stable Weil divisors} is the following
\begin{equation*}
    \mathcal{W}_T(X)=\left\langle D_{\rho } \mid \rho \in \Sigma (1)\right\rangle_{\Z} =
    \bigoplus_{\rho \in \Sigma (1)}\Z\cdot D_{\rho }
\end{equation*}
where $D_{\rho}=\overline{O}_{\rho}$ as above. Clearly the \emph{toric stable Cartier divisors} are given by $\mathcal{C}_T(X)=\mathcal{C}(X)\cap\mathcal{W}_T(X)$.

\begin{criterion}[Weil vs Cartier - \cite{Fulton} \S 3.3; \cite{Oda} Proposition 2.4]\label{cr:Cartier} Consider the toric stable Weil divisor $D=\sum_{\rho\in\Si(1)} a_{\rho} D_{\rho}\in \mathcal{W}_T(X)$. Then
\begin{equation*}
    D\in\mathcal{C}_T(X)\ \Longleftrightarrow\ \forall \s\in\Si\  \exists \u=\u(\s)\in M : \forall \rho\in\s(1)\ \langle\u,\n_{\rho}\rangle=-a_{\rho}
\end{equation*}
where $\n_{\rho}\in N$ is the unique generator of the semigroup $\rho\cap N$. Moreover if $\R\cdot\s = N_{\R}$ then such a $\u(\s)\in M$ is unique.
\end{criterion}

Actually it suffices to check the previous condition over the subset of \emph{maximal} cones in $\Si$. The given Criterion has then a number of interesting consequences:

\begin{corollary}\label{cor:fattorialità} Given $X=X(\Si)$ with $\Si$ \emph{simplicial} then $\mathcal{C}_T(X)\otimes\Q = \mathcal{W}_T(X)\otimes\Q$. Moreover if $\Si$ is \emph{non-singular}, meaning that $X$ is smooth, then $\mathcal{C}_T(X) = \mathcal{W}_T(X)$.
\end{corollary}

\begin{corollary} Let $\mathcal{P}(X)\subset\mathcal{W}(X)$ be the subgroup of \emph{principal divisors}. Then the morphism
\begin{equation*}
\begin{array}{llll}
div : & M & \longrightarrow & \mathcal{P}(X)\cap \mathcal{W}_{T}(X)=:
\mathcal{P}_{T}(X) \\
& \u & \longmapsto & div(\u):=\sum_{\rho \in \Sigma (1)}-\langle \u,\n_{\rho }\rangle
D_{\rho }
\end{array}
\end{equation*}
is surjective. Moreover if $\Si(1)$ generates the whole $N_{\R}$ then $div$ is also injective and $$M\stackrel{div}{\cong}\mathcal{P}_T(X)\ .$$
\end{corollary}

Recall that a toric variety $X(\Si)$ is compact if and only if $|\Si|=N_{\R}$. Then for any $\s\in\Si(n)$, $\R\cdot\s = N_{\R}$ and, by the Criterion \ref{cr:Cartier}, a Cartier divisor $D\in\mathcal{C}_T(X)$ uniquely determines a subset $v(\Si):=\{\u(\s)\in M\ |\ \s\in\Si(n)\}\subset M$ which turns out to be the set of vertexes of the associated (integral) polytope $\D_{D}$ defined in (\ref{politopo}).

\begin{criterion}[of (very)-ampleness - \cite{Oda} Theorem 2.13 and Corollary 2.14]\label{cr:ampiezza} Let $X(\Si)$ be a compact toric variety and $D\in\mathcal{C}_T(X)$ be a Cartier divisor. Then $D$ is ample if and only if $\Delta_D$ is a convex polytope and $v(\Si)$ consists of $|\Si(n)|$ distinct points of $M$. Moreover $D$ is very ample if and only if the previous conditions hold and
\begin{equation*}
    \forall\ \s\in\Si(n)\quad \D_D\cap M - \u(\s) \ \text{generates the semigroup $\s^{\vee}\cap M$}\ ,
\end{equation*}
where $\D_D\cap M - \u(\s)$ represents the set of integer points belonging to the polytope obtained by $\D_D$ after translating the vertex $\u(\s)$ into the origin of $M$.
\end{criterion}

Let $\Pic(X)$ be the group of line bundles modulo isomorphism. It is well known that for an \emph{irreducible} variety $X$ the map $D\mapsto\mathcal{O}_X(D)$ induces an isomorphism $\mathcal{C}(X)/\mathcal{P}(X)\cong\Pic(X)$. The Chow group of divisors is defined as the group of Weil divisors modulo rational (hence linear) equivalence, i.e. $A_{n-1}(X):=\mathcal{W}(X)/\mathcal{P}(X)$. Then the inclusion $\mathcal{C}(X)\subset\mathcal{W}(X)$ passes through the quotient giving an immersion $\Pic(X)\hookrightarrow A_{n-1}(X)$. One of main results on divisors on toric varieties is then the following

\begin{theorem}[\cite{Fulton} \S 3.4, \cite{CLS} Proposition 4.2.5]\label{thm:divisors} For a toric variety $X=X(\Sigma )$ the following sequence is exact
\begin{equation}
\def\objectstyle{\displaystyle}
\xymatrix@1{M \ar[r]^-{div} & *!U(.45){\bigoplus\limits_{\rho \in \Sigma (1)} \Z \cdot D_{\rho}}
\ar[r]^-d & A_{n-1} (X) \ar[r] & 0 }
\label{deg sequence}
\end{equation}
Moreover if $\Sigma (1)$ generates $N_{\R}$ then the morphism $div$ is injective giving the following exact sequences
\begin{equation}
\def\objectstyle{\displaystyle}
\xymatrix{
& 0 \ar[d] & 0 \ar[d] & 0 \ar[d] & \\
0 \ar[r] & M \ar[r]^{div}\ar[d]^{=} &
\mathcal{C}_T (X) \ar[r]\ar[d] & {\Pic(X)} \ar[r]\ar[d] & 0 \\
0 \ar[r] & M \ar[r]^-{div}\ar[d] & *!U(.40){\bigoplus\limits_{\rho \in \Sigma (1)} \Z \cdot D_{\rho}}
\ar[r]^-d & A_{n-1} (X) \ar[r] & 0 \\
 & 0 &  &  & }
\label{div-diagram}
\end{equation}
In particular $\Pic(X)$ and $A_{n-1}(X)$ turns out to be completely described by means of toric stable divisors and
\[
\rk \left( \Pic \left( X\right) \right) \leq \rk \left( A_{n-1}\left( X\right)
\right) =\left| \Sigma (1)\right| -n\ .
\]
Moreover if $\Si$ contains a $n$-dimensional cone then the first sequence in $(\ref{div-diagram})$ splits implying that $\Pic(X)$ is a free abelian group.
\end{theorem}

\subsubsection{Homogeneous coordinates and quotient spaces}\label{coord_omogenee} For each 1-dimensional cone $\rho\in\Si(1)$ introduce a variable $x_{\rho}$ and consider the polynomial ring
\begin{equation*}
    S = \C\left[ x_{\rho }:\rho \in \Sigma (1)\right]\ .
\end{equation*}
Notice that an effective divisor $D=\sum_{\rho }a_{\rho}D_{\rho}\in\mathcal{W}_T(X)$ determines an \emph{exponential monomial} $x^D:=\prod_{\rho }x_{\rho }^{a_{\rho }}$. Conversely, the \emph{logarithm} of a monomial $\prod_{\rho }x_{\rho }^{a_{\rho }}$ determines a divisor $D=\sum_{\rho}a_{\rho
}D_{\rho}=\log x^D$, giving a 1 to 1 correspondence between monomials of $S$ and effective toric invariant Weil divisors. Recalling the sequence (\ref{deg sequence}) it is then possible to give a \emph{grading} to the ring $S$ by setting
\begin{eqnarray*}
   \forall x^{D}\in S\quad \deg \left( x^{D}\right) &:=& d(D)\in A_{n-1}(X)  \\
   \forall \a\in A_{n-1}(X)\quad S_{\a}&:=&\bigoplus_{\deg \left( x^{D}\right) = \a} \C\cdot x^D
\end{eqnarray*}
so that the ring $S$ can be written as the direct sum $S=\bigoplus_{\a\in A_{n-1}(X)} S_{\a}$. Since $S_{\a}\cdot S_{\b}\subset S_{\a+\b}$ then $(S,\deg)$ is called the \emph{homogeneous coordinate ring of $X=X(\Si)$}, as introduced by D.~Cox in \cite{Cox}, to which the interested reader is referred for further details and proofs.

While the Chow group $A_{n-1}(X)$ and the graded ring $S$ are determined by the 1-skeleton $\Si(1)$, the higher dimensional cones of the fan $\Si$ share in determining a particular ideal of the polynomial ring $S$ called the \emph{irrelevant ideal $B$ of $X$}. Namely set
\begin{eqnarray}\label{B}
\nonumber
    \forall \s\in\Si\quad x^{\widehat{\sigma}} &:=& \prod_{\rho \in \Sigma (1)\setminus \sigma(1)}x_{\rho}\\
    B&:=&\left( x^{\widehat{\sigma }}:\sigma \in \Sigma \right)\ .
\end{eqnarray}
Clearly $B\subset S$ is an ideal and defines the following \emph{exceptional subset} of $\C^{\Si(1)}$
\begin{equation}\label{Z}
    Z:=\mathcal{V}(B)=\{x\in\C^{\Si(1)}\ |\ \forall\s\in\Si\quad x^{\widehat{\s}}=0\}\ .
\end{equation}
Finally let us apply the injective functor $\Hom_{\Z}\left(-,\C^{*}\right)$ to the right exact sequence (\ref{deg sequence}) to get the following left exact sequence of multiplicative groups
\begin{equation*}
    \xymatrix@1{1\ar[r] & \Hom_{\Z}\left(A_{n-1}(X), \C^{*}\right)
\ar[r]^-{d^{\vee}} & \Hom_{\Z}\left(\Z^{\left| \Sigma(1)\right| }, \C^{*}\right)
\ar[r]^-{{div}^{\vee}} & \Hom_{\Z}(M, \C^{*})}\ .
\end{equation*}
Define $G:=\Hom_{\Z}(A_{n-1}(X), \C^{*})$ and observe that $T:=\Hom_{\Z}(M, \C^{*})=N\otimes\C^*$ is the torus acting on $X$. Then the previous left exact sequence can be rewritten as follows
\begin{equation*}
    \xymatrix@1{1 \ar[r] & G \ar[r]^-{d^{\vee}} &
(\C^{*})^{\left|\Sigma(1)\right|} \ar[r]^-{{div}^{\vee}} & T}
\end{equation*}
defining an action of the group $G$ over $(\C^{*})^{|\Sigma(1)|}$ given by
\begin{equation}\label{azione}
    \begin{array}{lll}
G \times \C^{\left| \Sigma (1)\right| } & \longrightarrow & \C^{\left| \Sigma (1)\right| } \\
(g, (z_{\rho}))     & \longmapsto & d^{\vee}(g) \cdot (z_{\rho}):= (g(d(D_{\rho}))\ z_{\rho})\ .
\end{array}
\end{equation}
We are now in a position to state the following result of D.~Cox:

\begin{theorem}[\cite{Cox} Theorem 2.1]\label{thm:cox} Consider $X=X(\Si)$ and the associated exceptional subset $Z=\mathcal{V}(B)\subset\C^{|\Si(1)|}$ defined in $(\ref{Z})$ and $(\ref{B})$. Then $\C^{| \Sigma (1)| }\setminus Z$ is invariant under the action $(\ref{azione})$.

Moreover if $\Si(1)$ generates $N_{\R}$ then $X$ is the geometric quotient
\begin{equation}\label{quoziente}
    X\cong \left.\left(\C^{| \Sigma (1)| }\setminus Z\right)\right/G
\end{equation}
under the action $(\ref{azione})$ if and only if $\Si$ is simplicial.
\end{theorem}

\subsection{Hermite normal form}\label{HNF}
It is well known that Hermite algorithm provides an effective way to determine a basis of a subgroup of $\mathbb{Z}^n$. We briefly recall the definition and the main properties.
For details, see for example \cite{Cohen}.

\begin{definition}\label{def:HNF}
An $m\times n$ matrix $M=(m_{ij})$ with integral coefficients is in \emph{Hermite normal form} (abbreviated \emph{HNF}) if there exists $r\leq m$ and a strictly increasing map $f:\{1,\ldots,r\} \to \{1,\ldots,n\}$ satisfying the following properties:
\begin{enumerate}
\item For $1\leq i\leq r$, $m_{i,f(i)}\geq 1$, $m_{ij}=0$ if $j<f(i)$ and $0\leq m_{i,f(k)}< m_{k,f(k)}$ if $i<k$.
\item The last $m-r$ rows of $M$ are equal to $0$.
\end{enumerate}
\end{definition}

\begin{theorem}[\cite{Cohen} Theorem 2.4.3]\label{thm:cohen}
Let $A$ be an $m\times n$ matrix with coefficients in $\mathbb{Z}$. Then there exists a unique $m\times n$ matrix $B=(b_{ij})$ in HNF of the form $B=U\cdot A$ where $U\in\GL(m,\mathbb{Z})$.
\end{theorem}

We will refer to matrix $B$ as the HNF of matrix $A$. The construction of $B$ and $U$ is effective, see \cite[Algorithm 2.4.4]{Cohen}, based on Eulid's algorithm for greatest common divisor. In the following two applications of this algorithm will be considered: for computing a fan of a given wps (see Prop. \ref{prop:HNFdue}) and the so--called $Q$--canonical fan of $\P(Q)$ (see Prop. \ref{prop:fan minimale}). At this purpose, a key theoretical tool is the following (for the proof see \cite[\S 2.4.3]{Cohen})

\begin{proposition}\label{prop:HNFuno}
\begin{enumerate}
\item Let $L$ be a subgroup of $\mathbb{Z}^n$, $V=\{{\v}_1,\ldots,{\v}_m\}$ a set of generators, and let $A$ be the $m\times n$ matrix having ${\v}_1,\ldots,{\v}_m$ as rows. Let $B$ be the HNF of $A$. Then the non zero rows of $B$ are a basis of $L$.
\item Let $A$ be a $m\times n$ matrix, and let $B=U\cdot A^T$ be the HNF of the transposed of $A$, and let $r$ such that the first $r$ rows of $B$ are non zero. Then a $\mathbb{Z}$-basis for the kernel of $A$ is given by the last $m-r$ rows of $U$.
\end{enumerate}
\end{proposition}

\subsection{Transversion of a matrix}\label{trasversa} In the following, given a matrix $A\in\GL(n,\Q)$, the matrix obtained by taking the \underline{\emph{trans}}posed matrix of the in\underline{\emph{verse}} matrix
\begin{equation*}
    A^*:=((A)^{-1})^T
\end{equation*}
is called the \emph{transverse matrix} of $A$. We will see in the following (subsection \ref{fan_to_polytope}) that \emph{transversion} of a matrix resumes, up to the multiplication by a diagonal matrix of weights, the passage from a fan to a polytope (and back) associated with the same weighted projective space $\P(Q)$.

Here are some elementary properties of transversion:

\begin{proposition}\label{prop:trasversa} Let $A$ and $B$ be  matrices of $\GL(n,\Q)$. Then:
\begin{enumerate}
  \item $(A^*)^*=A$ i.e. transversion is an involution in $\GL(n,\Q)$,
  \item $(A\cdot B)^*=A^*\cdot B^*$,
  \item $\det(A^*)=1/\det(A)$,
  \item if $A$ is a superior (inferior) triangular matrix then $A^*$ is an inferior (superior) triangular matrix,
  \item if $A\in\GL(n,\Z)$ then $A^*\in\GL(n,\Z)$ too.
\end{enumerate}
\end{proposition}

\subsection{Weighted projective spaces} In the present subsection we will briefly recall the definition and some well known fact about \emph{weighted projective spaces} (\emph{wps} in the following). Proofs and details can be recovered in the extensive treatments \cite{Delorme} \cite{Mori}, \cite{Dolgachev} and \cite{BR}.

\begin{definition}
\label{def:WPS}
Set $Q:=\left(q_{0},\ldots ,q_{n}\right) \in \left( \N\setminus \{0\}\right)
^{n+1} $ and consider the multiplicative group $\mu_Q:=\mu_{q_{0}}\oplus \cdots \oplus \mu_{q_{n}}$
where $\mu_{q_i}$ is the group of $q_i$-th roots of unity.  Consider the following action of $\mu_Q$ over the $n$--dimensional complex projective space $\P^n$
\[
\begin{array}{cccc}
\mu _{Q}: & \mu_Q\times \P^{n} & \longrightarrow &
\P^{n} \\
& \left( \left( \zeta _{j}\right) ,\left[ z_{j}\right] \right) & \longmapsto
& \left[ \zeta _{j}z_{j}\right]\ .
\end{array}
\]
Let $\Delta _{Q}\subset \mu_Q$ be the diagonal subgroup and consider the quotient group $\W_Q:=\mu_Q/\Delta_Q$. Then the induced quotient space
\[
\P\left( Q\right) :=\P^{n}/\W _{Q}
\]
is called the $Q$--\emph{weighted projective space} ($Q$-wps).
\end{definition}

\begin{remark}
\label{rem:factorial-iso} If $q$ is the greatest common divisor of  $\left(
q_{0},\ldots ,q_{n}\right) $ then
\[
\Delta _{Q}\cong \mu_{q}
\]
Therefore we get the canonical isomorphism
\[
\P\left( Q\right) \cong \P\left( \frac{q_{0}}{q},\ldots ,
\frac{q_{n}}{q}\right)
\]
For this reason in the following \emph{we will always assume that}
\begin{equation*}
    q=\gcd\left( q_{0},\ldots ,q_{n}\right)=1\ .
\end{equation*}
\end{remark}

\begin{definition}[Weights vector] In the following a \emph{weights vector} $Q=(q_0,\ldots,q_n)$ will denote a $n+1$--tuple of \emph{coprime positive integer} numbers. Referring to notation defined in (\ref{lcm&GCD}), a weights vector $Q$ will be called \emph{reduced} if $d_{j}=1$, or equivalently $a_{j}=1$, for any $j=0,\ldots,n$.

\end{definition}

\begin{remark}
\emph{Every weighted projective space is a toric variety}. In fact the natural toric action over $\P^n$ passes through the quotient as follows
\[
\xymatrix{
(\C^{*})^n \times \P^{n}\ar[r]\ar[d]_{\tau_Q\times \pi_Q} & \P^{n}\ar[d]_{\pi_Q} \\
(\C^{*})^n \times \P^{n}(Q) \ar[r] & \P^n (Q)}
\]
where $\pi_Q$ is the natural quotient map and $\tau _{Q}$ is the quotient map associated with the action
\[
\begin{array}{ccc}
\mu_Q\times \left( \C^{*}\right) ^{n} & \longrightarrow &
\left( \C^{*}\right) ^{n} \\
\left( \left( \zeta _{j}\right) ,\left( t_{i}\right) \right) & \longmapsto &
\left( \zeta _{0}^{-1}\zeta _{i}t_{i}\right)
\end{array}
\]
Then the torus $\left( \C^{*}\right)
^{n}$\ can be embedded in $\P\left( Q\right) $ via the following map
\[
\begin{array}{ccc}
\left( \C^{*}\right) ^{n} & \hookrightarrow & \P\left(
Q\right) \\
\left( t_{1},\ldots ,t_{n}\right) & \longmapsto & \left[ 1:t_{1}:\ldots
:t_{n}\right]\
\end{array}
\]
whose image is the open subset $\P\left( Q\right) \setminus
\mathcal{V}\left( \prod_{j}z_{j}\right) $.\bigskip
\end{remark}

\subsection{A fan of $\P(Q)$}\label{ssez:fan}

 A fan of the wps $\P(Q)$ with $Q=(q_0,\dots,q_n)$ is presented in \cite{Fulton} at the end of \S 2.3.  Its construction is here recalled since it will be useful in the following.

\begin{proposition}[\cite{Fulton} \S 2.3]\label{prop:weighted-fan} Let $\{\e_1,\ldots,\e_n\}$ be a basis of the lattice $N$ and consider the following $n+1$ rational vectors
\[
\v_{0}:=-\frac{1}{q_{0}}\sum_{i=1}^n\e_i\quad,\quad \v_{i}:=\frac{\e_{i}}{q_{i}}\ .
\]
Let $^{Q}N$ be the lattice generated by $\v_0,\ldots,\v_n$ and, recalling (\ref{def:genfan}), consider $^{Q}\Sigma:=\fan(\v_0,\ldots,\v_n)$. Then $\P\left( Q\right)\cong X\left(^{Q}\Sigma\right)$.
\end{proposition}

\subsection{Divisors on $\P(Q)$}\label{ssez:divisori} The present subsection is devoted to apply Theorem \ref{thm:divisors} to describe the Chow and the Picard groups, and their generators, for the wps $\P(Q)$ with $Q=(q_0,\ldots,q_n)$. Let us first of all introduce the following notation
\begin{eqnarray}\label{lcm&GCD}
\nonumber
  d_{j} &:=& \gcd\left( q_{0},\ldots ,q_{j-1},q_{j+1},\ldots ,q_{n}\right)\ , \\
  a_{j} &:=& \lcm\left( d_{0},\ldots ,d_{j-1},d_{j+1},\ldots ,d_{n}\right)\ , \\
\nonumber
  a &:=& \lcm\left( a_{0},\ldots ,a_{n}\right)\ .
\end{eqnarray}
The weights vector $Q$ will be called \emph{reduced} if $d_{j}=1$, or equivalently $a_{j}=1$, for any $j=0,\ldots,n$.

\begin{proposition}
\label{prop:weight-rel.s} Since $\gcd(q_{0},\ldots ,q_{n})=1$, the following facts are true:
\begin{enumerate}
\item  $\gcd(q_{j},d_{j})=1$\ ,
\item  if $i\neq j$ then $\gcd(d_{i},d_{j})=1$\ ,
\item  $a_{j}\mid q_{j}$\ ,
\item  $\gcd(a_{j},d_{j})=1$\ ,
\item  $a_{j}d_{j}=a$\ ,
\item  setting $q_{j}':=q_{j}/a_{j}$, then $Q'=\left(
q_{0}',\ldots ,q_{n}'\right) $ is reduced; $Q'$ is then called the \emph{reduction of $Q$}.
\end{enumerate}
\end{proposition}
The proofs of these well known properties (see \cite{Dolgachev} 1.3.1) are elementary.

Instead we will prove the following property, which is no reported in the main treatments of the subject. This property will be fundamental to understand how fans and polytopes of wps behave when the reduction isomorphism given by Theorem \ref{thm:riduzione} is applied.

\begin{proposition}\label{prop:lcm} Let $Q=(q_0,\ldots,q_n)$ be a weights vector and $Q'=(q_{0}',\ldots ,q_{n}')$ be its reduction. Define
\begin{equation}\label{lcmQ&Q'}
    \d:=\lcm(q_0,\ldots,q_n)\quad\text{and}\quad \d':=\lcm(q_{0}',\ldots ,q_{n}')\ .
\end{equation}
Then $\d=a\d'$, where $a$ is defined in $(\ref{lcm&GCD})$.
\end{proposition}

\begin{remark} The Proposition \ref{prop:lcm} still holds when $q:=\gcd(q_0,\ldots,q_n)>1$. In fact we will not use this hypothesis in the following proof.
\end{remark}

\begin{proof}[Proof of Proposition \ref{prop:lcm}] It is easy to verify from definitions that
$$a,\delta'\hbox{ divide } \delta\hbox{ and } \delta\hbox{ divides } a\delta'.$$
It remains to prove that $a\delta'$ divides $\delta$. By definition of $a$ and $\delta'$ this amounts to show that $a_i$ divides $\frac\delta {q_k}a_k$ for every $i,k$, that is $d_j$ divides $\frac \delta {q_k}a_k$ for every $j,k$. If $j\not=k$ then $d_j$ divides $a_j$ and we are done; suppose now $j=k$, let $p$ be a prime dividing $d_k$ and let $p^t,p^r $ be the highest powers of $p$ dividing $d_k$ and $q_k$ respectively. Then $p^t$ divides $q_i$ for every $i\not=k$. If $r\geq t$, then $p^t$ divides $q_i$ for every $i$, so that $p^t$ divides $a_k$. If $r<t$ then $p^{t-r}$ divides $\frac\delta {q_k}$; moreover $p^r$ divides $d_i$ for every $i\not=k$, so that $p^r$ divides $a_k$. Therefore $p^t$ divides $\frac \delta{q_k} a_k$.
\end{proof}

The following statement resumes some well known fact, like the one that Picard and the Chow groups of a wps are free groups of rank 1, and some other fact, like a measure of \emph{non-factoriality} of a wps (see the following Remark \ref{rem:fattorialità}). We will give a complete proof of it, preceded by a proof of the following Lemma \ref{lm:condizioni}, since they will play a central role in the following.

\begin{theorem}\label{thm:P(Q)-divisori} Consider the wps $\P(Q)$ with $Q=(q_0,\ldots,q_n)$ and let $Q'=(q'_0,\ldots,q'_n)$ be the reduced weights vector. Then the following diagram is commutative and all rows and columns are exact
\begin{equation}
\xymatrix{
& 0 \ar[d] & 0 \ar[d] & 0 \ar[d] & \\
0 \ar[r] & ^Q M \ar[r]^-{div}\ar[d]^{\parallel} &
\mathcal{C}_T \left(\P(Q)\right) \ar[r]\ar[d] & \Pic \left(\P(Q)\right)
\ar[r]\ar[d] & 0 \\
0 \ar[r] & ^Q M \ar[r]^-{div}\ar[d] & \mathcal{W}_T \left(\P(Q)\right)
\ar[r]^-{d}\ar[d] & A_{n-1} (\P(Q)) \ar[r]\ar[d] & 0 \\
 & 0\ar[r] & \left.\Z\right/\d'\Z\ar[r]^{\cong}\ar[d] & \left.\Z\right/\d'\Z\ar[r]\ar[d] & 0\\
 && 0 & 0 & }
\label{W-div-diagram}
\end{equation}
$^QM$ is the dual lattice of the $n$-dimensional lattice $\ ^QN$ defined in Proposition \ref{prop:weighted-fan} and $\d'=\lcm(q'_0,\ldots,q'_n)$.

In particular:
\begin{enumerate}
  \item $A_{n-1}(\P(Q))\cong \Z$ is generated by the linear equivalence class of any divisor $\sum_{j=0}^n b_j D_j$, where $D_j$ is the toric invariant divisor associated with the 1-dimensional cone $\langle\v_j\rangle\in\Si(1)$ and $(b_0,\ldots,b_n)$ is any integer solution of the linear equation $\sum_{j=0}^n q'_j x_j=1$;
  \item $\Pic(\P(Q))\cong \Z$ is generated by the isomorphism class of any line bundle $\mathcal{O}_{\P(Q)}(\d' D)$, where $D$ represents a generator of $A_{n-1}(\P(Q))$; hence the immersion $\Pic(\P(Q))\hookrightarrow A_{n-1}(\P(Q))$ is the multiplication by $\d'$.
\end{enumerate}
\end{theorem}

\begin{remark}[Non-factoriality of $\P(Q)$]\label{rem:fattorialità} Recall that a variety $X$ is called \emph{factorial} if it is normal and every Weil divisor of $X$ is also a Cartier divisor i.e. the inclusion $\mathcal{C}(X)\subset\mathcal{W}(X)$ is actually an equality. Then the quotient group $\left.\mathcal{W}(X)\right/\mathcal{C}(X)$  gives a measure of how much $X$ is far form being factorial. The same can be done by tensoring with $\Q$ and looking at the quotient $\Q$-module $\left(\left.\mathcal{W}(X)\otimes\Q\right)\right/\left(\mathcal{C}(X)\otimes\Q\right)$.

\noindent The Corollary \ref{cor:fattorialità} of Criterion \ref{cr:Cartier}, states that every simplicial toric variety is $\Q$-factorial. Moreover it states that a smooth toric variety is factorial. The vertical exact sequences in (\ref{W-div-diagram}) give a converse of the latter fact, for a wps. Namely, since $\left.\Z\right/\d'\Z=0$ if and only if $\d'=1$, meaning that we are dealing with the usual projective space $\P^n$, we get that:
\begin{itemize}
  \item [(a)] \emph{a wps $\P(Q)$ is factorial if and only if $\P(Q)=\P^n$}\ ,
  \item [(b)] \emph{if $Q\neq(1,\ldots,1)$ then $\P(Q)$ is singular and $\P(Q)\not\cong\P^n$}
\end{itemize}
Moreover $\d'$ qualifies \emph{a geometric invariant} of $\P(Q)$ guaranteeing that
\begin{equation*}
    \d'(Q_1)\neq\d'(Q_2)\ \Longrightarrow\ \P(Q_1)\not\cong\P(Q_2)\ .
\end{equation*}
\end{remark}

\begin{lemma}\label{lm:condizioni} Let $Q=(q_0,\ldots,q_n)$ be a weights vector; let $\{\v_0,\ldots,\v_n\}$ be a set of vectors in $\mathbb{Q}^n$, generating $\mathbb{Q}^n$ and such that $\sum_{j=0}^{n}q_{j}\v_{j}=0$. Let $L$ be the lattice generated in $\mathbb{Q}^n$ by $\{\v_0,\ldots,\v_n\}$ and $L'$ be the sublattice generated by  $\{q_0\v_0,\ldots,q_n\v_n\}$. Then the following properties hold:
\begin{itemize}
  \item [(a)]
   $ \left[L:L'\right]=\prod_{j=0}^n q_j$;
  \item [(b)]let $V:=(v_{ij})$ be the $n\times (n+1)$ matrix whose columns are given by components of $\v_0,\ldots,\v_n$ over a basis $\e_{1},\ldots,\e_{n}$ of $L$ i.e. $\v_j=\sum_{i=1}^n = v_{ij}\e_i$, for every $j=0,\ldots,n$, and denote by $V_j$ the $n$-minor of $V$ obtained by deleting the $j$-th column as in (\ref{Vj}). Then
      \begin{equation*}
        \forall\ j=0,\ldots,n\quad V_j=(-1)^{\epsilon +j}q_j\ ,\ \text{for a fixed $\epsilon\in\{0,1\}$}\ ,
      \end{equation*}
  \item [(c)] $\forall\ j=0,\ldots,n\quad \v_j=d_j\n_j$\ , where $\n_j$ is the generator of the semigroup $\langle\v_j\rangle\cap\ L$ and $d_j$ is defined in (\ref{lcm&GCD}); in particular $L$ is the lattice generated by $\{\n_0,\ldots,\n_n\}$; moreover $\{\n_0,\ldots,\n_n\}$ satisfy the hypotheses of this Lemma with respect to the reduced weights vector $Q'$ i.e. they generate $\mathbb{Q}^n$ and $\sum_{j=0}^{n}q'_{j}\n_{j}=0$.

\end{itemize}
\end{lemma}

\begin{proof} For (a), observe that $L'$ has $q_1\v_1,\ldots,q_n\v_n$ as a basis. Then $L'$ has index $\prod_{j=1}^nq_j$ in the lattice $L_0$ generated by $\v_1,\ldots,\v_n$. The quotient $L/L_0$ is cyclic generated by the image of $\v_0$, so that $[L:L_0]$ divides $q_0$. If $r\v_0\in L_0$, with $r\in\Z$ then $r\v_0=\sum_{j=1}^ns_j\v_j$ with $s_1,\ldots,s_n\in\Z$. Since $\gcd(q_0,\ldots,q_n)=1$ then  there exists $\lambda\in\Z$ such that $r=-\lambda q_0$, $s_i=\lambda q_i$ for $i=1,\ldots,n$; in particular $q_0$ divides $r$, so that $[L:L_0]=q_0$ and $[L:L']=[L:L_0][L_0:L']=\prod_{i=0}^n q_i$.\par
\noindent (b): for $j=0,\ldots,n$, let $L_j$ be the lattice generated by $\v_0,\ldots,\v_{j-1},\v_{j+1},\ldots,\v_n$. Then $|V_j|=[L:L_j]=q_j$, as we have shown in (a) for the case $j=0$. Let $\epsilon\in\{0,1\}$ be such that $V_0=(-1)^{\epsilon}q_0$. Then
\begin{equation*}
    \forall\ j=0,\ldots,n\quad V_j=(-1)^{j}\frac{q_j}{q_0}V_0=(-1)^{\epsilon +j}q_j
\end{equation*}
since $\sum_{j=0}^{n}q_{j}\v_{j}=0$. \par
\noindent (c):  we have
\begin{equation*}
    \forall\ j=0,\ldots ,n\quad q_{j}\v_{j}=-\sum_{k\neq j}q_{k}\v_{k}=-d_{j}\sum_{k\neq j}\tilde{q}_{k}\v_{k}
\end{equation*}
where $\tilde{q}_{k}:=q_{k}/d_{j}\in\N$. By (1) in Proposition \ref{prop:weight-rel.s}, $\gcd(q_j,d_j)=1$ meaning that
\begin{equation*}
    \forall\ j=0,\ldots ,n\quad \exists\ \v_{j}'\in L\ :\ \v_{j}=d_{j}\v_{j}'\ .
\end{equation*}
Then (5) in Proposition \ref{prop:weight-rel.s} allows to write
\begin{equation}\label{a-primo}
    0=\sum_{j=0}^{n}q_{j}\v_{j}=\sum_{j=0}^{n}\left( q_{j}'a_{j}\right)
\left( d_{j}\v_{j}'\right) =a\sum_{j=0}^{n}q_{j}'\v_{j}'\Longrightarrow \sum_{j=0}^{n}q_{j}'\v_{j}'
=0\ .
\end{equation}
Moreover $\v'_0,\ldots\v'_n$ generate $L$ and (a) ensures that the following index
\begin{equation}\label{b-primo}
    \left[L:\langle q_{j}'\v_{j}'\ |\ j=0,\ldots,n\rangle\right] = \prod_{j=0}^n q'_j.
\end{equation}
 Then the proof ends up by showing that, for all $j$, $\v'_j=\n_j$. At this purpose consider $h_j\in\N$ such that $\v'_j=h_j\n_j$\ . If $V'=(v'_{ij})$ is the matrix of components of $\v'_0,\ldots,\v'_n$, over the basis $\e_{1},\ldots,\e_{n}$ of $L$, then
\begin{equation*}
    \forall\ j=0,\ldots,n\quad \left| V_{j}'\right| =q'_{j}\
\end{equation*}
On the other hand
\begin{equation*}
 \v_{0}'\in h_0 L\Longrightarrow \forall\ i=1,\ldots,n\quad h_{0}\mid
v_{i0}'\Longrightarrow \forall\ k=1,\ldots,n\quad h_{0}\ |\ \left|
V_{k}'\right| =q_{k}'\ .
\end{equation*}
Therefore (6) in Proposition \ref{prop:weight-rel.s} implies that
\[
h_{0}\ |\ \gcd\left( q_{1}',\ldots ,q_{n}'\right)
=1\Longrightarrow h_{0}=1
\]
Analogously $h_{j}=1$, for all $1\leq j\leq n$. Hence $\v'_j=\n_j$\ .
\end{proof}

An alternative proof of point (c) in the previous Lemma \ref{lm:condizioni} can be found on the one hand in \cite{Fujino1}, Prop. 2.3 and on the other hand in \cite{BB} Prop. 2. 

\begin{proof}[Proof of Theorem \ref{thm:P(Q)-divisori}] First of all let us recall that $\P(Q)=X\left(^Q\Si\right)$ where $^Q\Si$ is the fan described in Proposition \ref{prop:weighted-fan}. Then $\left|^{Q}\Sigma (1)\right| =n+1$ and the sequence (\ref{deg sequence}) turns out to be also left exact, since $\left|^{Q}\Sigma (1)\right|$ generates $N_{\R}$, giving
\begin{equation*}
    \def\objectstyle{\displaystyle}
\xymatrix@1{0 \ar[r] & ^Q M \ar[r]^-{div} &
\bigoplus_{j=0}^{n}\Z\cdot D_{j} \ar[r]^-{d} &
A_{n-1}\left( \P (Q) \right) \ar[r] & 0}
\end{equation*}
where $D_j$ is the toric invariant divisor associated with the 1-dimensional cone $\langle\v_j\rangle\in\Si(1)$.
Then
\begin{equation}\label{div_u}
    \forall\ \u\in \ ^{Q}M\quad div\left( \u\right)
=-\sum_{j=0}^{n}\left\langle u,\n_{j}\right\rangle \ D_{j}
\end{equation}
where $\n_j$ is the generator of the semigroup $\langle\v_j\rangle\cap\ ^QN$. Let $^{Q}\e_{1},\ldots
,^{Q}\e_{n}$ be a basis of $^{Q}N$ and $^{Q}\e_{1}^{\vee},\ldots ,^{Q}\e_{n}^{\vee}$ be the dual basis of $^{Q}M=\Hom
\left(^{Q}N,\Z\right) $. If $\n_j=\sum_{i=1}^n n_{ij}\ ^{Q}\e_{j}$ then (\ref{div_u}) gives that
\begin{eqnarray}\label{ker d}
\im \ \left( div\right) =\ker \left( d\right) &=&\left\langle
div\left( \ ^{Q}\e_{i}^{\vee}\right)\ |\ i=1,\ldots ,n\right\rangle \\
&=&\left\langle \sum_{j=0}^{n}n_{ij}D_{j}\ |\ i=1,\ldots ,n\right\rangle   \nonumber
\end{eqnarray}
which is a free subgroup of rank $n$ in $\mathcal{W}_{T}(X)$, since $^{Q}\Sigma $ is simplicial. Then the quotient $\left.\mathcal{W}_T(\P(Q))\right/\ker (d)$ gives
\begin{equation*}
    A_{n-1}\left( \P\left( Q\right) \right) \cong \left\langle \
D_{0},\ldots ,D_{n}\ |\ \forall i=1,\ldots,n\quad \sum_{j=0}^{n}n_{ij}\
D_{j}=0\right\rangle
\end{equation*}
which turns out to consists of a free part of rank 1 and a possible torsion part. The latter vanishes if the basis $\{div(^Q\e^{\vee}_i)\ |\ 1\leq i\leq n\}$ of the $\Z$-module $\ker(d)$ can be extended to get a basis of the $\Z$-module $\mathcal{W}_T(\P(Q))$ i.e. if there exists a divisor $\sum_{j=0}^n b_j D_j\in\mathcal{W}_T(\P(Q))$ such that
\begin{equation}\label{no-torsione}
    \det \left(
\begin{array}{cccc}
n_{10} & n_{11} &
\cdots & n_{1n} \\
\vdots & \vdots &  & \vdots \\
n_{n0} & n_{n1} &
\cdots & n_{nn} \\
b_{0} & b_{1} & \cdots & b_{n}
\end{array}
\right) =1
\end{equation}
Recall now the last part of point (c) in Lemma \ref{lm:condizioni}. Then
\begin{eqnarray*}
    \det\left(
          \begin{array}{c}
            n_{ij} \\
            \hline b_j\\
          \end{array}
        \right) &=&\det \left(
\begin{array}{cccc}
-\sum_{h=1}^{n}\frac{q_{h}'}{q_{0}'}n_{1h} & n_{11} &
\cdots & n_{1n} \\
\vdots & \vdots &  & \vdots \\
-\sum_{h=1}^{n}\frac{q_{h}'}{q_{0}'}n_{nh} & n_{n1} &
\cdots & n_{nn} \\
b_{0} & b_{1} & \cdots & b_{n}
\end{array}
\right) \\
    &=& \left( -1\right) ^{n+2}\left( b_{0}+\frac{q_{1}'}{q_{0}'}
b_{1}+\cdots +\frac{q_{n}'}{q_{0}'}b_{n}\right) \left|
\begin{array}{ccc}
n_{11} & \cdots & n_{1n} \\
\vdots &  & \vdots \\
n_{n1} & \cdots & n_{nn}
\end{array}
\right|
\end{eqnarray*}
Since $\left|
\begin{array}{ccc}
n_{11} & \cdots & n_{1n} \\
\vdots &  & \vdots \\
n_{n1} & \cdots & n_{nn}
\end{array}
\right|=\pm q'_0$, up to replace $b_j$ with $-b_j$, we get
\begin{equation*}
    \det\left(
          \begin{array}{c}
            n_{ij} \\
            \hline b_j\\
          \end{array}
        \right) = \sum_{j=0}^{n}q_{j}'b_{j}
\end{equation*}
and condition (\ref{no-torsione}) can be rewritten as follows
\begin{equation}\label{no-torsione2}
    \sum_{j=0}^{n}q_{j}'b_{j}=1\ .
\end{equation}
The existence of an integer solution $(b_0,\ldots,b_n)$ of equation (\ref{no-torsione2}) is guaranteed by recalling that $\gcd\left(q_{0}',\ldots ,q_{n}'\right) =1$. In particular
\begin{equation*}
    A_{n-1}\left( \P\left( Q\right) \right) =\Z\cdot \left(
\sum_{j=0}^{n}b_{j}d\left(D_{j}\right) \right) \cong \Z
\end{equation*}
proving (1) in the statement of Theorem \ref{thm:P(Q)-divisori}.

By Criterion \ref{cr:Cartier}, a toric invariant Weil divisor $\sum_{j=0}^n c_j D_j\in\mathcal{W}_T(\P(Q))$ is Cartier if and only if for any maximal cone $\s\in\ ^Q\Si$ there exists $\u=\u(\s)\in\ ^QM$ such that
\begin{equation}\label{cartier?}
    \forall\ j : \langle\n_j\rangle\in\s(1)\quad \langle\u,\n_j\rangle=-c_j
\end{equation}
Maximal cones of $^Q\Si$ are precisely given by all the cones $\s\in\Si(n)$, which are $n+1$ and can be enumerated by the unique $j$ such that $\v_j\not\in\s(1)$. Then the Criterion \ref{cr:Cartier} can be applied as follows to the case of $^Q\Si$
\begin{equation}\label{cartier?2}
    \sum_{j=0}^n c_j D_j\ \text{is Cartier}\ \Longleftrightarrow\ \forall\ l=0,\ldots,n\quad \exists \u_l\in\ ^QM\ :\ \forall\ j\neq l\quad \langle\u_l,\n_j\rangle=-c_j\ .
\end{equation}
Let us apply (\ref{cartier?2}) to a representative divisor $D:=\sum_{j=0}^n b_jD_j$ of a generator of the Chow group, i.e. such that its coefficients $b_j$ give an integer solution of the linear equation (\ref{no-torsione2}). By writing $\u_{l}=\sum_{i=1}^{n}u_{il}\ ^{Q}\e_{i}^{\vee}$, we get $\langle\u_l,\n_j\rangle=\sum_{i=1}^{n}u_{il}n_{ij}$, meaning that, for all $0\leq l\leq n$, we are looking for a solution of the $n\times n$ linear system
\begin{equation}\label{sistema}
    \forall\ j\neq l\quad \sum_{i=1}^{n}n_{ij}u_{il} = -b_j\ .
\end{equation}
By Lemma \ref{lm:condizioni}, $|\det(n_{ij}\ |\ j\neq l)|= q'_l$, which in general is greater then 1, implying that the linear system (\ref{sistema}) admits rational but not integral solutions i.e. $D$ \emph{is not a Cartier divisor} except for the case $Q=(1,\ldots,1)$ giving the usual projective space $\P^n$. Moreover a positive multiple $kD$ is Cartier if $q'_l\ |\ k$ for any $l=0,\dots,n$. This means that the minimum positive multiple of $D$ which is Cartier is obtained by setting $k=\d':=\lcm(q'_0,\ldots,q'_n)$. Then $\Pic(\P(Q)))\subset A_{n-1}(\P(Q))=\Z\cdot d(D)$ is the subgroup $\Z\cdot (\d'd(D))$, ending up the proof of (2) in the statement of Theorem \ref{thm:P(Q)-divisori}.

Finally (\ref{W-div-diagram}) follows by applying (1) and (2) to the diagram (\ref{div-diagram}) in Thm. \ref{thm:divisors}.
\end{proof}

\subsection{A polytope of $\P(Q)$}

The present subsection is inspired by an observation due to I.~Dolgachev (\cite{Dolgachev} 1.2.5).

\begin{proposition}\label{prop:1-politopo} Define $\delta:=\lcm(q_1,\ldots,q_n)$, with $\gcd(q_1,\ldots,q_n)=1$, and let $\D$ be the $n$-dimensional simplex obtained as the convex hull of the origin and the $n$ points of $M\otimes\R\cong\R^n$
\begin{equation*}
    p_1=\left(\frac{\delta}{q_1},0,\ldots,0\right)\ ,\  \ldots\ ,\ p_n=\left(0,\dots,0,\frac{\delta}{q_n}\right)
\end{equation*}
If $Q=(1,q_1,\ldots,q_n)$ then there exists a very ample divisor $H$ of $\P(Q)$ whose polarization gives $(\P(Q),H)\cong \P_{\D}$.
\end{proposition}

\begin{proof} Since $q_0=1$, Proposition \ref{prop:weighted-fan} gives $\v_0=-\sum_{i=1}^n q_i\v_i$ and $\P(Q)\cong X(^Q\Si)$ where $^Q\Si=\fan(\v_0,\ldots,\v_n)\subset\ ^QN$. If $\n_j$ is the unique generator of the semigroup $\langle\v_j\rangle\cap\ ^QN$ then Lemma \ref{lm:condizioni},(c) gives that $\n_j=d_j\v_j=\v_j$, since $\gcd(q_1,\ldots,q_n)=1$ implies that $Q=(1,q_1,\ldots,q_n)$ is reduced. Let $D_0$ be the toric invariant divisor  associated with the 1-dimensional cone generated by $\v_0$ and $\D:=\D_{\d D_0}\subset\ ^QM\otimes\R$ be the integral polytope associated with $\d D_0$ as in (\ref{politopo}). The facets of $\D_{\d D_0}$ are lying on the following $n+1$ hyperplanes of $\R^n$
\begin{eqnarray*}
    \langle \u,\v_0\rangle = -\d\ &\Leftrightarrow&\ \sum_{i=1}^n q_i x_i =\d \\
    \langle \u,\v_i\rangle = 0\ &\Leftrightarrow&\ x_i=0\ ,\quad\text{for $1\leq i\leq n$}
\end{eqnarray*}
whose $n$ by $n$ intersections give precisely vertexes $p_1,\ldots,p_n$ in the statement. The proof ends up by applying the following Lemma \ref{lm:moltampiezza}.
\end{proof}

\begin{lemma}\label{lm:moltampiezza} In the same notation as Proposition \ref{prop:1-politopo}, the divisor $H:=\d D_0$ of $\P(Q)$ is very ample.
\end{lemma}

\begin{proof} It is an application of the Criterion \ref{cr:ampiezza}. First of all observe that the set of vertexes of $\D:=\D_{\d D_0}$ is given by
\begin{equation*}
    v\left(\ ^Q\Si\right)= \left\{0,p_1,\ldots,p_n\right\}
\end{equation*}
immediately implying the ampleness of $\d D_0$. Let us now denote by $\s_j\in\ ^Q\Si(n)$ the cone generated by $\{\v_0,\ldots,\v_n\}\setminus\{\v_j\}$ and consider the dual cone $\s^{\vee}_j\subset\ ^QM\otimes\R$. Then
\begin{eqnarray*}
  \s^{\vee}_0 &=& \langle\v_1^{\vee},\ldots,\v_n^{\vee}\rangle \\
  \forall\ i=1,\ldots,n\quad \s^{\vee}_i&=&\left\langle\{-\v^{\vee}_i\}\cup\left\{\frac{\d}{q_k}\v^{\vee}_k-\frac{\d}{q_i}\v^{\vee}_i\ |\ \forall k\neq i\right\}\right\rangle
\end{eqnarray*}
where $\{\v_1^{\vee},\ldots,\v_n^{\vee}\}\subset\ ^QM$ is the dual basis of $\{\v_1,\ldots,\v_n\}\subset\ ^QN$. Then, for any vertex $\u(\s_j)\in v\left(\ ^Q\Si\right)$, observe that $\D\cap M - \u(\s_j)$ generates the semigroup $\s_j^{\vee}\cap M$. In fact
\begin{itemize}
  \item if $j=0$ then $\D\cap M - \u(\s_0)=\D\cap M$ which generates the semigroup $\langle\v_1^{\vee},\ldots,\v_n^{\vee}\rangle\cap M$,
  \item if $j=i\neq 0$ then the translated polytope $\D- \u(\s_i)$ has vertexes given by
  \begin{eqnarray*}
    0 &=& p_i - \u(\s_i) \\
    p_{0i} &=& 0 - \u(\s_i) = -p_i\\
    \forall\ k\neq i\quad p_{ki} &=& p_k - \u(\s_i) = p_k - p_i\ .
  \end{eqnarray*}
\end{itemize}
Then $\D\cap M -\u(\s_i)$ clearly generates the semigroup $\s_i^{\vee}\cap M$.
\end{proof}

\begin{remark}\label{rem:politopo} Consider the polytope $\D=\D_{\d D_0}$ assigned by Proposition \ref{prop:1-politopo}. Every facet of $\D$ admitting the origin as a vertex corresponds to a polytope whose associated polarized toric variety is a sub-wps of $(\P(Q),H)$. Namely the opposite facet $\D_i$ to the vertex $p_i$ is the polytope of $(\P(Q_i),H_{i})\subset(\P(Q),H)$ where $Q_i=(1,q_1,\ldots,q_{i-1},q_{i+1},\ldots,q_n)$, $H_{i}:=\d (D_0\cap\P(Q_i))$ and $\D_i=\D_{H_{i}}$. Clearly $H_i$ is a very ample divisor of $\P(Q_i)$.

The opposite facet $\D_0$ to the origin turns out to be a polytope whose associated toric variety is the sub-wps $(\P(Q_0),H_0)\subset(\P(Q),H)$ where $Q_0=(q_1,\ldots,q_n)$ and $H_0$ is a suitable polarization. This fact can be checked by observing that the translated polytope $\D_0-\u(\s_1)$ is associated as in (\ref{politopo}) with the divisor $\d\left/q_1\right. (D_1\cap\P(Q))$, where $D_1$ is the toric invariant divisor  associated with the 1-dimensional cone generated by $\v_1$. Notice that, by Theorem \ref{thm:P(Q)-divisori}, $\d D_0$ and $\left(\d\left/q_1\right.\right)\,D_1$ are linear equivalent divisors of $\P(Q)$. Hence $\left(\d\left/q_1\right.\right)\,D_1$ is a very ample divisor of $\P(Q)$ implying that its section $H_0:=\d\left/q_1\right. (D_1\cap\P(Q_0))$ is a very ample divisor of $\P(Q_0)$ giving
$$\left(\P(Q_0),H_0\right)\cong\left(\P_{\D_{\d\left/q_1\right. (D_1\cap\P(Q))}},\mathcal{O}(1)\right) = \left(\P_{\D_0-\u(\s_1)},\mathcal{O}(1))\cong(\P_{\D_0},\mathcal{O}(1)\right)\ .$$
We are now able to produce the polytope of a general wps $\P(Q)$, without the restriction $q_0=1$.
\end{remark}

\begin{corollary}[\cite{Dolgachev} 1.2.5]\label{cor:politopo} Let $Q=(q_0,\ldots,q_n)$ be a weights vector and let $\D$ be the $(n+1)$-dimensional simplex obtained as the convex hull of the origin and the $n+1$ points of $M\otimes\R\cong\R^{n+1}$
\begin{equation*}
    P_0=\left(\frac{\delta}{q_0},0,\ldots,0\right)\ ,\  \ldots\ ,\ P_n=\left(0,\dots,0,\frac{\delta}{q_n}\right)
\end{equation*}
Let $\D_{\pi}:=\D\cap\pi$ be the $n$-dimensional polytope in $(M\cap\pi)\otimes\R$ obtained by intersecting with the hyperplane $\pi:=\mathcal{V}(\sum_{j=0}^n q_j x_j - \delta)\subset\R^{n+1}$. Then there exists a suitable very ample divisor $H_{\pi}$ such that $(\P(Q),H_{\pi})\cong (\P_{\D_{\pi}},\mathcal{O}(1))$.
\end{corollary}

\begin{proof}
Consider the $(n+1)$-dimensional wps $\P(1,Q)$. Proposition \ref{prop:1-politopo} ensures that $(\P(1,Q),H)\cong(\P_{\D},\mathcal{O}(1))$ where $\D$ is the $(n+1)$-dimensional simplex described in the statement and $H$ is a suitable polarization. Therefore the polytope of $\P(Q)$ is given by the opposite facet of $\D$ with respect to the origin, as described in Remark \ref{rem:politopo}, which is precisely the facet cut out from $\D$ by the hyperplane $\pi$.
\end{proof}

\subsection{The reduction theorem} The following \emph{reduction Theorem} is a well known fact (\cite{Delorme} \S 1,\cite{Dolgachev} 1.3.1) we are now able to completely prove by the toric point of view, as an application of the previous Theorem \ref{thm:P(Q)-divisori} and by giving a quotient presentation of $\P(Q)$, following the Cox Theorem \ref{thm:cox}. Compare also with \cite{Fujino1} \S 2.2.

\begin{theorem}[Reduction Theorem]\label{thm:riduzione} Let $Q'=(q'_0,\ldots,q'_n)$ be the reduced weights vector of $Q=(q_0,\ldots,q_n)$. Then
\begin{equation}\label{weighted quot}
\P\left( Q\right) \cong
\left.\left( \C^{n+1}\setminus \left\{ 0\right\} \right) \right/\C^{*} \cong \P\left( Q'\right)
\end{equation}
where the quotient is realized by means of the (reduced) action
\begin{equation*}
    \begin{array}{cccc}
\nu _{Q'}: & \C^{*}\times \C^{n+1} & \longrightarrow
& \C^{n+1} \\
& \left( t,\left( z_{j}\right) \right) & \longmapsto & \left(
t^{q_{0}'}z_{0},\ldots ,t^{q_{n}'}z_{n}\right)\ .
\end{array}
\end{equation*}
\end{theorem}

\begin{proof} Let $S=\C[x_0,\ldots,x_n]$ be the polynomial ring obtained by associating the variable $x_j$ with the 1-dimensional cone $\langle\v_j\rangle\in\Si(1)$. We want to describe the grading induced on $S$ by proceeding as in \ref{coord_omogenee}. The kernel of the degree map $d:\mathcal{W}_T(\P(Q))\rightarrow A_{n-1}(\P(Q))$ is described in (\ref{ker d}), implying that
\begin{equation}\label{relazioni}
    \forall i=1,\ldots,n\quad d\left( \sum_{j=0}^{n}n_{ij}D_{j}\right)
=\sum_{j=0}^{n}n_{ij}d\left(D_{j}\right) =0\ .
\end{equation}
By  Lemma \ref{lm:condizioni},(c), $n_{i0}=-\sum_{k=1}^{n}\left( q_{k}'/q_{0}'\right) n_{ik}$. Then the multiplication by $q'_0$ of (\ref{relazioni}) can be rewritten as follows
\begin{equation}\label{relazioni2}
    \forall\ i=1,\ldots,n\quad \sum_{k=1}^{n}n_{ik}\left( q_{0}'d\left(D_{h}\right) -q_{k}'d\left(D_{0}\right) \right) =0\ .
\end{equation}
Moreover Lemma \ref{lm:condizioni} gives that $|\det(n_ik)|=q'_0\neq 0$. Therefore equations (\ref{relazioni2}) imply that
\begin{equation}\label{deg x_h}
\forall\ k=1,\ldots,n\quad q_{0}'d\left(D_{k}\right) =q_{k}'d\left(D_{0}\right) \ .
\end{equation}
Let $d(D)=\sum_{j=0}^n b_j d(D_j)$ be a generator of $A_{n-1}(\P(Q))$, which is that $\sum_{j=0}^n q'_jb_j=1$, by Theorem \ref{thm:P(Q)-divisori},(2). Then $\deg(d(D))=1$ and (\ref{deg x_h}) gives
\begin{eqnarray*}
    q'_0&=&q'_0\cdot \deg(d(D)) = \deg\left(\sum_{j=0}^n b_j q'_0 d(D_j)\right)\\
    &\stackrel{\text{(\ref{deg x_h})}}{=}& \left(\sum_{j=0}^n q'_jb_j\right) \deg\left(d(D_0)\right) = \deg\left(d(D_0)\right)
\end{eqnarray*}
and, again by (\ref{deg x_h}), we get
\begin{equation}\label{graduazione}
    \forall\ j=0,\ldots,n\quad \deg(x_j)=\deg\left(d(D_j)\right)=q'_j
\end{equation}
which defines a grading on $S$.

 Let us now use Theorem \ref{thm:cox} to obtain a quotient description of $\P(Q)$. First of all observe that the acting group is $G:=\Hom(A_{n-1}(\P(Q)),\C^*)\cong\C^*$, whose action on $\C^{|^Q\Si(1)|}$, as defined in (\ref{azione}), is given by
 \begin{equation}\label{azione_PQ}
    \begin{array}{lll}
\C^* \times \C^{n+1} & \longrightarrow & \C^{n+1} \\
(t, (z_j))     & \longmapsto & d^{\vee}(t) \cdot (z_j):= (t^{\deg x_j}\ z_j)\stackrel{(\ref{graduazione})}{=}(t^{q'_j}z_j)\ .
\end{array}
 \end{equation}
This is precisely the action $\nu_{Q'}$ in the statement. The irrelevant ideal $B$ defined in (\ref{B}) is given by $B=(x_0,\ldots,x_n)\subset S$, then
\begin{equation}\label{Z-PQ}
    Z=\mathcal{V}(B)=\{0\}\subset\C^{n+1}\ .
\end{equation}
Therefore, by (\ref{azione_PQ}) and (\ref{Z-PQ}), Theorem \ref{thm:cox} gives
\begin{equation}\label{PQ-quoziente}
    \P(Q)\cong\left.\left(\C^{n+1}\setminus\{0\}\right)\right/\C^*
\end{equation}
where the quotient is realized by means of the action $\nu_{Q'}$, as defined in (\ref{azione_PQ}).

On the other hand $\P(Q')=X\left(^{Q'}\Si\right)$ where $^{Q'}\Si:=\fan(\n_0,\ldots,\n_n)$ in the lattice $^{Q'}N:=\langle\n_0,\ldots,\n_n\rangle$. Therefore the kernel of the map $d:\mathcal{W}_T(\P(Q'))\rightarrow A_{n-1}(\P(Q'))$ is still described by (\ref{ker d}), then still giving (\ref{relazioni}), (\ref{relazioni2}), (\ref{deg x_h}) and the grading (\ref{graduazione}) on the generators of $A_{n-1}(\P(Q'))$. Then the application of Theorem \ref{thm:cox} allows to conclude that
\begin{equation}\label{PQ'-quoziente}
    \P(Q')\cong\left.\left(\C^{n+1}\setminus\{0\}\right)\right/\C^*
\end{equation}
since $\Hom(A_{n-1}(\P(Q')),\C^*)\cong\C^*$ and its action is still given by (\ref{azione_PQ}), when is observed that the exceptional subset $Z=\mathcal{V}(B)$ remains that described in (\ref{Z-PQ}). Then (\ref{PQ-quoziente}) and (\ref{PQ'-quoziente}) prove (\ref{weighted quot}) and the statement of Theorem \ref{thm:riduzione}.
\end{proof}

\subsection{Ampleness and very ampleness of divisors on $\P(Q)$}

In the present subsection we will extend to the class of weighted projective spaces a property of smooth complete toric varieties, which is the coincidence of ampleness and very ampleness concepts for divisors (\cite{Oda} Corollary 2.15).

\begin{proposition}\label{prop:molta-ampiezza} With the same notation as in Theorem \ref{thm:P(Q)-divisori}, if $D$ is any representative of a generator of $A_{n-1}(\P(Q))$ then $\d' D$ is a very ample divisor. Moreover for a positive integer $k$, the multiple $k D$ is an ample divisor if and only if $\d'$ divides $k$. In particular a divisor of $\P(Q)$ is ample if and only if it is very ample.
\end{proposition}

\begin{remark} Consider the projective embedding $\P(Q)\hookrightarrow\P^r$ associated with the very ample divisor $\d' D$ defined in Proposition \ref{prop:molta-ampiezza}. Then $\P^r=\P(|\mathcal{O}_{\P(Q)}(\d' D)|)$ and $\mathcal{O}_{\P(Q)}(\d' D)\cong\mathcal{O}_{\P(Q)}(H)$, where $H$ is the divisor cut out on $\P(Q)$ by a hyperplane section. By Theorem \ref{thm:P(Q)-divisori},(2) this line bundle generates $\Pic(\P(Q))$. Then it makes sense to fix the following
\begin{notation} $\mathcal{O}_{\P(Q)}(1)$ is the line bundle generating $\Pic(\P(Q))$, which is that
\begin{equation}\label{notazione}
    \forall\ D=\sum_{j=0}^n a_j D_j\in\mathcal{W}_T(\P(Q)) : \sum_{j=0}^n a_jq'_j=\d'\quad    \mathcal{O}_{\P(Q)}(1)\cong\mathcal{O}_{\P(Q)}(D)\ .
\end{equation}

\end{notation}
\end{remark}

\begin{proof}[Proof of Proposition \ref{prop:molta-ampiezza}] By Theorem \ref{thm:P(Q)-divisori},(2) the line bundle  $\mathcal{O}_{\P(Q)}(\d' D)$ generates $\Pic(\P(Q))$. On the other hand, by the Reduction Theorem \ref{thm:riduzione}, $\P(Q)\cong\P(Q')$ where $Q'$ is the reduction of $Q$. This isomorphism sends $D$ to a representative $D'$ of a generator of $A_{n-1}(\P(Q'))$. Hence $\mathcal{O}_{\P(Q)}(\d' D')$ generates $\Pic(\P(Q'))$: the very ampleness of $\d' D$ can then be proven by showing that $\d'D'$ is very ample. At this purpose recall Proposition \ref{prop:1-politopo} and Corollary \ref{cor:politopo} and consider the $(n+1)$-dimensional polytope $\D$ defined as the convex hull of the origin and the $n+1$ points in $\R^{n+1}$
\begin{equation*}
    P_0=\left(\frac{\delta'}{q'_0},0,\ldots,0\right)\ ,\  \ldots\ ,\ P_n=\left(0,\dots,0,\frac{\delta'}{q'_n}\right)\ .
\end{equation*}
With the same notation given in Remark \ref{rem:politopo}, there exists a very ample divisor $H$ on $\P(1,Q')$ such that $(\P(1,Q'),H)\cong(\P_{\D},\mathcal{O}(1))$ and, calling $\D_0$ the facet of $\D$ which is opposite to the origin, one gets $(\P(Q'),H_0)\cong(\P_{\D_0},\mathcal{O}(1))\cong(\P_{\D_0-P_0},\mathcal{O}(1))$. Let $D_0$ be the toric invariant divisor associated with the 1-dimensional cone of the fan of $\P(1,Q')$ corresponding to the weight $q'_0$: then $\D_0-P_0=\D_{(\d'/q'_0) D_0}$. By Remark \ref{rem:politopo} $(\d'/q'_0) D_0$ turns out to be a very ample divisor of $\P(1,Q')$. Therefore, cutting out by means of the sub-wps $\P(Q')\subset\P(1,Q')$, the divisor $\d'/q'_0 (D_0\cap\P(Q'))$ is very ample on $\P(Q')$. Since Theorem \ref{thm:P(Q)-divisori},(1) gives the linear equivalence $\d'/q'_0 (D_0\cap\P(Q'))\equiv\d' D'$ we get that $\d'D'$ is also a very ample divisor of $\P(Q')$.

Let $D$ be a divisor whose linear equivalence class $d(D)\in A_{n-1}(\P(Q))$ is a generator i.e. $A_{n-1}(\P(Q))\cong\Z\cdot d(D)$. As already observed in the proof of Theorem \ref{thm:P(Q)-divisori}, a positive multiple $kD$ is Cartier if $q'_l\ |\ k$ for any $l=0,\dots,n$. Since an ample divisor is Cartier, assuming $kD$ ample gives that $\d'|k$. Conversely, since $\d' D$ is a very ample divisor, if $\d'\,|\,k>0$ then $kD$ is very ample, hence ample.

To prove the last assertion in the statement of Proposition \ref{prop:molta-ampiezza} we have to show that ampleness implies very ampleness. At this purpose consider an ample divisor $E$ of $\P(Q)$.  Then it is Cartier and $d(E)=k d(D)$ for some positive multiple $k$ of $\d'$. Since $\d' D$ is very ample, this means that $E$ is a very ample divisor.
\end{proof}

\section{Characterization of fans giving $\P(Q)$}\label{sez: fan}

In the present section we will answer to the following nested questions:
\begin{itemize}
  \item[(I)] \emph{given a weights vector $Q$ and a fan $\Si$, when $\P(Q)\cong X(\Si)$?}
  \item[(II)] \emph{given a fan $\Si$, when there exists a weights vector $Q$ such that $\P(Q)\cong X(\Si)$?}
\end{itemize}
We will give necessary and sufficient conditions for $\Si$ to be associated with the complete toric variety $\P(Q)$.

\subsection{Characterizing the fan}

Let $Q=(q_0,\ldots,q_n)$ be a weights vector and consider a $n$-dimensional lattice $N$ and a subset of $n+1$ integer vectors $\{\v_0,\ldots,\v_n\}\subset N$. Let us consider the 1--generated fan
\begin{equation*}
    \Si:=\fan(\v_0,\ldots,\v_n)
\end{equation*}
and \emph{the} associated fan matrix $V=(\mathbf{v}_0,\ldots,\mathbf{v}_n)$, where the columns order is fixed by the weights order in $Q$. Then the fan $\Si$ and the matrix $V$ are associated each other.

\begin{theorem}\label{thm:fan} Consider the fan $\Sigma=\fan(\v_0,\ldots,\v_n)$ and the associated matrix $V=(\v_0,\ldots,\v_n)$. Then the following facts are equivalent:
\begin{enumerate}
  \item $\Si$ is a fan of $\P(Q)$\ ,
  \item $\sum_{j=0}^{n}q_{j}\v_{j}=0$ and the sub-lattice $N':=\langle q_0\v_0,\dots,q_n\v_n\rangle\subset N$ has finite index
  \[
    [N:N']=\prod_{j=0}^n q_j\ ,
  \]
  \item $\forall\ j=0,\ldots,n\quad V_j=(-1)^{\epsilon +j}q_j\ ,\ \text{for $\epsilon\in\{0,1\}$}\ ,$
  \item $q_{0}\v_{0}=-\sum_{i=1}^{n}q_{i}\v_{i}$ and $|V_0|:=|\det(\v_1,\ldots,\v_n)|=q_0$\ .
\end{enumerate}
\end{theorem}

\begin{proof} $(1)\Rightarrow(2)$. Recalling Proposition \ref{prop:weighted-fan}, this is precisely the content of Lemma \ref{lm:condizioni},(a).

$(2)\Rightarrow(3)$. This is Lemma \ref{lm:condizioni},(b).

$(3)\Rightarrow(4)$. For any $k=1,\ldots,n$ consider the $(n+1)\times(n+1)$ matrix
\begin{equation*}
    A_k:=\left(
           \begin{array}{ccc}
             v_{k0} & \cdots & v_{kn} \\
             \hline \\
              & V &  \\
           \end{array}
         \right)\ .
\end{equation*}
Since the first and the $(k+1)$-th rows of $A_k$ are equal we get
\begin{equation*}
    \forall\ k=1,\ldots,n\quad 0=\det(A_k)=\sum_{j=0}^n (-1)^j v_{kj} V_j \stackrel{(3)}{=} (-1)^{\epsilon} \sum_{j=0}^n q_jv_{kj} \ \Rightarrow\ \sum_{j=0}^n q_j\v_{j}=0\ .
\end{equation*}

$(4)\Rightarrow(2)$. Since $|V_0|=q_0$ then $\{q_1\v_1,\ldots,q_n\v_n\}$ is a basis of the sub-lattice $N'$. Hence
\begin{equation*}
    [N:N']=|\det(q_1\v_1,\ldots,q_n\v_n)| = \left(\prod_{i=1}^nq_{i}\right) |V_{0}|\stackrel{(4)}{=}\prod_{j=0}^n q_j\ .
\end{equation*}

$(2)\Rightarrow(1)$. First of all notice that (2) guarantees that $\Si$ is simplicial and $\Si(1)$ generates $N\otimes\R$. Then Theorem \ref{thm:cox} can be applied to give a geometric quotient description of $X(\Si)$. Then Lemma \ref{lm:condizioni},(c) gives that
\begin{itemize}
  \item \emph{$\forall\ j=0,\ldots,n\quad \v_j=d_j\n_j$\ , where $\n_j$ is the generator of the semigroup $\langle\v_j\rangle\cap N$ and $d_j$ is defined in (\ref{lcm&GCD}); in particular $\n_0,\ldots,\n_n$ satisfy the condition \emph{(2)} with respect to the reduced weights vector $Q'$ i.e.
      \begin{itemize}
      \item [($2'$)] $\sum_{j=0}^{n}q'_{j}\n_{j}=0$\ and the sub-lattice $N''':=\langle q'_0\n_0,\dots,q'_n\n_n\rangle\subset N'':=\langle\n_0,\ldots,\n_n\rangle$ has finite index
  \[
    \left[N'':N'''\right]=\prod_{j=0}^n q'_j\ .
  \]
      \end{itemize}}
\end{itemize}
It is then possible to reproduce here the proof of Theorem \ref{thm:riduzione} which exhibits the toric variety $X(\Si)$ as the following geometric quotient
\begin{equation*}
    X(\Si)\cong\left.\left(\C^{n+1}\setminus\{0\}\right)\right/\C^*
\end{equation*}
where the quotient is realized by means of the action $\nu_{Q'}$. Then $X(\Si)\cong\P(Q')\cong\P(Q)$.
\end{proof}

\begin{remark}\label{rem:Fn,Vn} The previous Theorem \ref{thm:fan} answers Question (I) opening the present section. For Question (II) notice that a matrix $V=(\v_0,\ldots,\v_n)\in \Mat(n,n+1,\Z)$ gives immediately an associated \emph{pseudo-weights} vector $Q:=(|V_0|,|V_1|,\ldots,|V_n|)$ which is the vector of absolute values of maximal minors of $V$. It turns out to be an actual weights vector if every maximal minor of $V$ does not vanish. Therefore the answer to Question (II) can be formulated as follows
\begin{itemize}
  \item \emph{given a fan $\Si=\fan(\v_0,\ldots,\v_n)\subset N\otimes\R$ then there exists a weights vector $Q=(q_0,\ldots,q_n)$ such that $X(\Si)\cong\P(Q)$ if and only if the following conditions hold:}
      \begin{enumerate}
        \item \emph{the matrix $V=(\v_0,\ldots,\v_n)$ admits only non vanishing coprime maximal minors i.e.  $\forall\ j=0,1,\dots,n\quad V_j \neq 0$ and $\gcd(V_j\ |\ 0\leq j\leq n)=1$;}
        \item \emph{the columns $\v_j$ of $V$ satisfy one of the equivalent conditions (2), (3), (4) of Theorem \ref{thm:fan} with respect to the weights $q_j:=|V_j|$}.
      \end{enumerate}
\end{itemize}
\end{remark}

\begin{definition}[$F$--admissible matrices]\label{def:Fn,Vn} A matrix $V\in\Mat(n,n+1;\Z)$ will be called \emph{$F$--admissible} if it satisfies conditions (1) and (2) in Remark \ref{rem:Fn,Vn}. The subset of $F$--admissible matrices will be denoted by $\mathfrak{V}_n\subset\Mat(n,n+1;\Z)$.
\end{definition}

The natural action of the permutation group $\mathfrak{S}_{n+1}$ on $\Mat(n,n+1;\Z)$, defined by exchanging columns, clearly restricts to $\mathfrak{V}_n$. Under the natural embedding of $\mathfrak{S}_{n+1}$ as a subgroup of $\GL(n+1,\Z)$, such an action is represented by right multiplication. Let $\mathfrak{F}(Q)$ be the set of fans in $N_{\R}=N\otimes\R$ defining $\P(Q)$ and define $\mathfrak{F}_n:=\bigcup_Q \mathfrak{F}(Q)$. Then the previous answer to Question (II) rewrites as follows:
\begin{equation}\label{Fn}
    \mathfrak{F}_n\cong \mathfrak{V}_n/\mathfrak{S}_{n+1}
\end{equation}
where the quotient is taken under right multiplication.
All these considerations give rise to the following algorithm for \emph{recognizing a wps fan}:

\begin{algorithm}[Recognizing a fan of $\P(Q)$]\label{alg:fan} See 1.1 and 1.2 in \cite{RTpdf,RT-Maple} for a Maple implementation.
\begin{itemize}
\item Input: a fan $F=({\bf v}_0,...,{\bf v}_n)$.
\item Let $V:=Mat({\bf v}_0,...,{\bf v}_n)$ be the associated matrix and set $q_i:=|V_i|$;
\item If $q_i\not =0$ for $i=0,...,n$ and $\gcd(q_0,...,q_n)=1$ and $\sum_{i=0}^n q_i{\bf v}_i=0$ then $F$ is a fan of $\P(Q)$ with $Q=(q_0,...,q_n)$; else it is not a fan of a wps.
\item Output: either the weights vector $Q$ or an error message .
\end{itemize}
\end{algorithm}

\subsection{Hermite normal form of weights and fans of $\P(Q)$}\label{ssez:HNFfan}
In the subsection \ref{ssez:fan} we already exhibit a fan of $\P(Q)$, as proposed by Fulton in his book \cite{Fulton}. Actually by an algorithmic point of view that fan is not very explicitly presented. The following result, which is a direct consequence of Theorem \ref{thm:fan}, gives rise to a surprising method to get a fan of a given wps $\P(Q)$, which turns out to be \emph{encoded} in the switching matrix giving the HNF of the transposed weights vector $Q^T$. Since the latter is obtained by a well known algorithm, based on Eulid's algorithm for greatest common divisor  \cite[Algorithm 2.4.4]{Cohen}, this gives an alternative method to produce a fan of $\P(Q)$ (see the following Algorithm \ref{alg:daQaF}) which can be performed by any procedures computing elementary linear algebra operations.

\begin{proposition}\label{prop:HNFdue} Let $Q=\left(q_0,\ldots,q_n\right)$ be a weights vector, $B$ the HNF of the transposed vector $Q^T$ and $U\in \GL({n+1},\mathbb{Z})$ be such that $U\cdot Q^T=B$. Let $C$ be the matrix consisting of the last $n$ rows of $U$ and let ${\v}_j$ be the $j^{\rm th}$ column vector of $C$, for $0\leq j\leq n$. Let $L,L'$ be the lattices generated in $\mathbb{Z}^n$ by ${\v}_0,\ldots,{\v}_n$ and $q_0{\v}_0,\ldots,q_n{\v}_n$ respectively. Then
\begin{enumerate}
\item $B=\left (\begin{array}{l} 1\\0\\ \vdots\\ 0\end{array}\right)$;
\item $L=\mathbb{Z}^n$;
\item $\sum_{j=0}^n q_j{\v}_j=0$;
\item for all  $0\leq j\leq n$ there exists $\epsilon\in\{0,1\}$ such that $C_j=(-1)^{\epsilon+j}q_j$.
\item $[L:L']=\prod_{j=0}^n q_j$.
\end{enumerate}
As a consequence of Theorem \ref{thm:fan}, $\fan(\v_0,\dots,\v_n)$ is a fan of $\P(Q)$.
\end{proposition}

\begin{proof}
The rank of $Q$ is 1, so by definition of HNF, $B=\left (\begin{array}{l} a\\0\\ \vdots\\ 0\end{array}\right)$ with $a\geq 1$. By the equality $Q^T=U^{-1}\cdot B$ we see that $a$ must divide $q_0,\ldots,q_n$, so that $a=1$; this proves (1).
By point (2) of Proposition \ref{prop:HNFuno} the rows of $C$ are a basis of the cotorsion-free subgroup of $\mathbb{Z}^{n+1}$ defined by the equation $\sum_{j=0}^n q_jx_j=0$. Therefore (3) holds. Moreover the elementary divisor Theorem assures that there exist matrices $A_1\in \GL({n+1},\mathbb{Z})$ and $A_2\in \GL(n,\mathbb{Z})$ such that $A_1\cdot C\cdot A_2 = \left(\begin{array}{llllll} 1 & 0&0&\ldots&0&0\\
0 & 1&0&\ldots&0&0\\ &&\vdots &&&\\
0&0&0&\ldots& 1&0
\end{array}
\right)$. By transposing we see that the columns of $C$ generate $\mathbb{Z}^n$ and this proves (2). Point (4) is immediate from point (1) using the Cramer rule. Moreover by (3) we see that $L'$ is generated by $q_1{\v}_1,\ldots,q_n{\v}_n$: therefore it has index $\prod_{i=1}^nq_i$ in the lattice generated by ${\v}_1,\ldots,{\v}_n$, and the latter has index $q_0$ in $\mathbb{Z}^n$ by (4); point (5) follows.
\end{proof}

\begin{algorithm}[Producing a fan of $\P(Q)$]\label{alg:daQaF} See 0.b in \cite{RTpdf,RT-Maple} for a Maple implementation.
\begin{itemize}
\item Input: the weights vector $Q=(q_0,\ldots,q_n)$.
\item Apply the HNF algorithm in order to find a matrix $U\in{\rm GL}_{n+1}(\mathbb{Z})$ such that
$U\cdot \left( \begin{array}{l} q_0\\q_1\\\vdots\\q_n \end{array}\right )= \left( \begin{array}{l} 1\\0\\\vdots \\0 \end{array}\right )$.
\item Delete from $U$ the first row and take the list $F$ of the column vectors of the resulting matrix.
\item Output $F$.
\end{itemize}
\end{algorithm}

\subsection{Equivalence of fans giving the same wps}\label{matrici e fan}

After \cite{Oda} \S 1.5, it is well known that two toric varieties are isomorphic if and only if their fans are related by a unimodular transformation of the ambient lattice. In case of weighted projective spaces this means that \emph{two fans $\Si=\fan(\v_0,\ldots,\v_n)$ and $\Si'=\fan(\v'_0,\ldots,\v'_n)$, in a $n$-dimensional lattice $N$, describe the same (up to isomorphism) wps $\P(Q)$ if and only if there exists $A\in\GL(n,\Z)$ and a permutation $\s\in\mathfrak{S}_{n+1}$ such that}
\begin{equation*}
    \forall\ 0\leq j\leq n\quad  A\cdot \v_j =\v'_{\s(j)}\ .
\end{equation*}
Since $\mathfrak{S}_{n+1}$ naturally embeds as a subgroup of $\GL(n+1,\Z)$, this is equivalent to say that
\begin{equation*}
    X(\Si)\cong \P(Q)\cong X(\Si')\ \Leftrightarrow\ \begin{array}{c}
                                                       \exists\ A\in\GL(n,\Z) \\
                                                       \exists B\in\mathfrak{S}_{n+1}\subset\GL(n+1,\Z)
                                                     \end{array}
\ :\ V'=A\cdot V\cdot B
\end{equation*}
where $V=(\v_0,\ldots,\v_n)$ and $V'=(\v'_0,\ldots,\v'_n)$ are associated matrices with fans $\Si$ and $\Si'$, respectively. By recalling the definition of $\mathfrak{F}(Q)$, $\mathfrak{F}_n$ and $\mathfrak{V}_n$ given in Definition \ref{def:Fn,Vn}, define
\begin{eqnarray*}
  \mathfrak{F}_n^{\text{red}} &:=& \bigcup \{\mathfrak{F}(Q)\ |\ Q\ \text{is reduced}\} \\
  \mathfrak{V}_n^{\text{red}} &:=& \{V\in\mathfrak{V}_n\ |\ \text{$(|V_0|,\ldots,|V_n|)$ is reduced}\} \ .
\end{eqnarray*}
Then we get the following:

\begin{proposition}\label{prop:classification} The following sets are bijectively equivalent
\begin{eqnarray}
    \label{1:1}
    \left\{\text{wps's}\right\}/\text{iso}&\stackrel{1:1}{\longleftrightarrow}&\left\{Q\in(\N\setminus\{0\})^{n+1}\ |\ Q\ \text{is reduced}\right\}/\mathfrak{S}_{n+1}\\
    \nonumber
    &\stackrel{1:1}{\longleftrightarrow}&\GL(n,\Z)\left\backslash\mathfrak{F}_n^{\text{red}}\right.\\
    \nonumber
    &\stackrel{1:1}{\longleftrightarrow}&\GL(n,\Z)\left\backslash\mathfrak{V}_n^{\text{red}}\right/\mathfrak{S}_{n+1}\ .
\end{eqnarray}
\end{proposition}

\begin{proof} For the first bijection in (\ref{1:1})
 let us define, at first, a map
\begin{equation}\label{biezione0}
    \left\{\text{wps's}\right\}/\text{iso}\longrightarrow\left\{Q\in(\N\setminus\{0\})^{n+1}\ |\ Q\ \text{is reduced}\right\}/\mathfrak{S}_{n+1}
\end{equation}
by sending the isomorphism class of $\P(Q)$ to the $\mathfrak{S}_{n+1}$--orbit of the reduced weights vector $Q'$.

\noindent First of all one has to prove that the map in (\ref{biezione0}) is well defined which is that $\P(Q)\cong\P(P)$ implies that the reduced weights vectors $Q'$ and $P'$ are related by a permutation. In fact the given isomorphism of wps's gives an isomorphisms of polarized varieties $\left(\P(Q),\mathcal{O}_{\P(Q)}(1)\right)\cong\left(\P(P),\mathcal{O}_{\P(P)}(1)\right)$ which is a toric equivariant isomorphism preserving characteristic points of cones (defined as in (\ref{caratteristico})) and their toric orbits. Then toric invariant divisors $\{D_0,\ldots,D_n\}$ generating $\mathcal{W}_T(\P(Q))$ are sent, up to a permutation, into toric invariant divisors $\{E_0,\ldots,E_n\}$ generating  $\mathcal{W}_T(\P(P))$. Recalling notation (\ref{notazione}), any Cartier divisor $D$ representing $\mathcal{O}_{\P(Q)}(1)$ can be written as $D=\sum_{j=0}^n b_jD_j$ where $(b_0,\ldots,b_n)$ is any integer point of the hyperplane $\pi_Q:=\mathcal{V}\left(\sum_{j=0}^n q'_jx_j-\d_Q'\right)$ defined in $\R^{n+1}$ by letting $Q'=(q'_0,\dots,q'_n)$ be the reduced weights vector of $Q$ and $\d'_Q:=\lcm(q'_0,\dots,q'_n)$. Analogously any Cartier divisor $E$ representing $\mathcal{O}_{\P(P)}(1)$ can be written as $E=\sum_{j=0}^n c_jE_j$ where $(c_0,\ldots,c_n)$ is any integer point of the hyperplane $\pi_P:=\mathcal{V}\left(\sum_{j=0}^n p'_jx_j-\d_P'\right)$ defined in $\R^{n+1}$ by letting $P'=(p'_0,\dots,p'_n)$ be the reduced weights vector of $P$ and $\d'_P:=\lcm(p'_0,\dots,p'_n)$. The given isomorphism of polarized varieties then sends any representative $D$ of $\mathcal{O}_{\P(Q)}(1)$ in a representative $E$ of $\mathcal{O}_{\P(P)}(1)$ and viceversa, meaning that every integer point of $\pi_Q$ is, up to a permutation, an integer point of $\pi_P$ and viceversa. Therefore, up to a permutation on the coordinates $x_j$, $\pi_Q\equiv\pi_P$, which is that $Q'=P'$ up to a permutation, as expected.

\noindent Since the map in (\ref{biezione0}) is clearly surjective, it remains to prove that it is also injective. In fact consider $\P(Q_1)$ and $\P(Q_2)$ such that the reduced weights vectors $Q'_1$ and $Q'_2$ are related by a permutation. The the Reduction Theorem \ref{thm:riduzione} immediately gives that $\P(Q_1)\cong\P(Q_2)$, as expected.

The second bijection in (\ref{1:1}), is obtained by applying the Reduction Theorem \ref{thm:riduzione} to Theorem 1.13 in \cite{Oda}.

Finally, the third bijection in (\ref{1:1}) is obtained immediately by applying (\ref{Fn}).
\end{proof}

\subsection{A $Q$--canonical fan of $\P(Q)$}\label{ssez:Q-fan}

In subsections \ref{ssez:fan} and \ref{ssez:HNFfan} we already presented two ways to produce a fan of a given wps. In the present subsection we want to use the characterization (4) in Theorem \ref{thm:fan} to present a further result returning a $Q$--\emph{canonical fan} of $\P(Q)$, in the sense that the associated fan matrix is in HNF, up to a permutation of columns (see the following Remark \ref{rem:Q-canonico}). This fact presents the fan in a triangular shape and generated by as much as possible number of vectors $\e_1,\ldots,\e_n$ in a given basis of the lattice $N$. Moreover it turns out to be a convenient procedure to get a fan of $\P(Q)$ by hands (see the following Example \ref{ex:Q-canonico}).

\noindent The following Proposition \ref{prop:fan minimale} has to be compared with results in $\S 3$ and $\S 4$ of \cite{Conrads}. In particular points from (1) to (3) give a rewrite of Proposition 3.2 and Remark 3.3 in \cite{Conrads}. Point (4) has to be compared with Remark 4.3 and Theorem 4.5 in \cite{Conrads}. Anyway proofs are different: Conrads proved his results by algebraic considerations, while here we will obtain the same results as a consequence of Theorem \ref{thm:fan} and of the methods used to prove it, allowing us to remain in the toric setup.

\begin{proposition}\label{prop:fan minimale} Let $Q=(q_0,\ldots,q_n)$ be a weights vector. For any $j$ with $1\leq j\leq n$, define $k_j:=\gcd(q_0,q_j,q_{j+1},\ldots,q_n)$.
Then:
\begin{enumerate}
  \item $k_j\ |\ k_{j+1}$\ ,
  \item either $k_n=(q_0,q_n)=1$ or there exists a positive integer $i$, with $1\leq i\leq n-1$, such that $k_i=1$ and $k_{i+1}>1$,
  \item consider a superior triangular matrix $V^0=(\v_1,\ldots,\v_n)\in\GL(n,\Z)$ whose columns $\v_j$ are such that:
  \begin{eqnarray*}
    \forall\ 1\leq j\leq i-1\quad \v_j &=& \e_j \\
    \forall\ i\leq j\leq n-1\quad v_{jj} &=& \left.k_{j+1}\right/ k_{j} \\
                                    v_{nn} &=& \left. q_0\right/k_n
  \end{eqnarray*}
  where $v_{kj}$ is the $k$-th entry of the column $\v_j$; then there exists a choice for $v_{kj}$ with $i\leq k\leq j$ such that $V^0$ can be completed to a matrix $V=(\v_0,\v_1,\ldots,\v_n)\in \Mat(n,n+1;\Z)$ whose columns satisfy the following condition
  \begin{equation*}
    \sum_{j=0}^n q_j \v_j = 0\ ;
  \end{equation*}
  in particular the columns of $V$ satisfy condition (4) of Theorem \ref{thm:fan}, hence generate a fan of $\P(Q)$.
  \item there exists a unique choice of the previous matrices $V$ and $V^0$ such that $V^0$ is in HNF with only non negative entries; then the column $\v_0$ in $V$ admits only negative entries. Moreover the matrix $V'$, obtained by exchanging each other the columns $\v_0$ and $\v_n$ in $V$, is in HNF.
\end{enumerate}
\end{proposition}

\begin{proof} (1) is obvious. (2) follows from (1) by recalling the hypothesis
\begin{equation*}
 k_1 = \gcd(q_0,q_1,\ldots,q_n)=1 \ .
\end{equation*}
To prove (3) we have to show the existence of an integral vector $\v_0=\left(\begin{array}{c}
                                                                                    v_{10} \\
                                                                                    \vdots \\
                                                                                    v_{n0}
                                                                                  \end{array}\right)
$ satisfying the following equations
\begin{eqnarray}
    \label{eq.i}
    \forall\ 1\leq j\leq i-1 &\quad& q_0 v_{j0} + q_j + \sum_{k=i}^n q_k v_{jk} = 0 \\
    \label{eq.>i}
    \forall\ i\leq j\leq n-1 &\quad& q_0 v_{j0} + q_j\frac{k_{j+1}}{k_j} + \sum_{k=j+1}^n q_{k}v_{jk}=0 \\
    \label{eq.n}
     && q_0 v_{n0} + q_n\frac{q_0}{k_n} = 0\ .
\end{eqnarray}
The last equation (\ref{eq.n}) is clearly satisfied by putting $v_{n0}=-q_n/k_n=-q_n/(q_0,q_n)$.
The $j$-th equation in (\ref{eq.>i}) admits integer solutions for $(v_{j0},v_{j,j+1},\ldots,v_{jn})$ if and only if
$$\gcd(q_0,q_{j+1},\ldots,q_n)=k_{j+1}\ |\ q_j\frac{k_{j+1}}{k_j}$$
which is clearly true since $k_j\ |\ q_j$, by definition. Finally the $j$-th equation in (\ref{eq.i}) admits integer solutions for $(v_{j0},v_{ji},v_{j,i+1}\ldots,v_{jn})$ if and only if
$$\forall\ 1\leq j\leq i-1\quad \gcd(q_0,q_i,\ldots,q_n)=k_{i}\ |\ q_j$$
which is clearly true since $k_i=1$, by the previous point (2). Recall now that $V^0$ is a triangular matrix, giving
\begin{equation*}
    \det (V^0) = \prod_{j=1}^n v_{jj} = k_{i+1}\cdot\frac{k_{i+2}}{k_{i+1}}\cdots\frac{k_n}{k_{n-1}}\cdot\frac{q_0}{k_n} = q_0
\end{equation*}
which is enough to get condition (4) of Theorem \ref{thm:fan} for the columns of $V$.

\noindent To prove (4) let us first of all observe that, for any $1\leq j\leq i-1$, $\v_{j}=\e_j$, meaning that the first $i-1$ columns of $V^0$ are composed of nonnegative entries satisfying the HNF conditions. Moreover $V^0$ is upper triangular. Then it remains to prove that there exists a unique choice for $v_{jk}$ such that
\begin{equation*}
    \forall k: i\leq k\leq n\ ,\ \forall j: j<k \quad 0\leq v_{jk}< v_{kk}\ .
\end{equation*}
The $j$--th equation in (\ref{eq.>i}) can be rewritten as follows
\begin{equation*}
    q_0v_{j0}+q_nv_{jn}=-q_j\frac{k_{j+1}}{k_j}-\sum_{k=j+1}^{n-1} q_k v_{jk}\ .
\end{equation*}
Fixing variables $v_{jk}$, for $j+1\leq k\leq n-1$, the previous diophantine equation admits solutions for $v_{j0},v_{jn}$ if and only if
\begin{equation}\label{div-condizione}
    k_n=\gcd(q_0,q_n)\ |\ -q_j\frac{k_{j+1}}{k_j}-\sum_{k=j+1}^{n-1} q_k v_{jk}\ .
\end{equation}
Moreover, given a particular solution $v_{jn}^{(0)}$, all the possible integer solutions for $v_{jn}$ are given by
\begin{equation*}
    v_{jn}=v_{jn}^{(0)} - \frac{q_0}{k_n}\cdot h_{jn} = v_{jn}^{(0)} - v_{nn}\cdot h_{jn}\ ,\quad \forall\ h_{jn}\in\Z\ .
\end{equation*}
Divide $v_{jn}^{(0)}$ by $v_{nn}$. Then the remainder of such a division gives a unique choice for $v_{jn}$ such that
\begin{equation*}
    \forall\ i\leq j\leq n-1 \quad 0\leq v_{jn}<v_{nn}\ .
\end{equation*}
Analogously the $j$-th equation in (\ref{eq.i}) can be rewritten as follows
\begin{equation*}
    q_0v_{j0}+q_nv_{jn}=-q_j-\sum_{k=i}^{n-1} q_k v_{jk}
\end{equation*}
and the same argument ensures the existence of a unique choice for $v_{jn}$ such that
\begin{equation*}
    \forall\ 1\leq j\leq i-1 \quad 0\leq v_{jn}<v_{nn}\ .
\end{equation*}
Then the last column in $V^0$ can be uniquely chosen with non-negative entries satisfying the HNF condition. Iteratively, condition (\ref{div-condizione}) is satisfied if and only if there exist integer solutions for $x,v_{jk}$ in the diophantine equation
\begin{equation*}
    k_n x +q_{n-1} v_{j,n-1}= -q_j\frac{k_{j+1}}{k_j}-\sum_{k=j+1}^{n-2} q_k v_{jk}
\end{equation*}
which is if and only if
\begin{equation*}
    \gcd(k_n,q_{n-1})=\gcd(q_0,q_{n-1},q_n)=:k_{n-1}\ |\ -q_j\frac{k_{j+1}}{k_j}-\sum_{k=j+1}^{n-2} q_k v_{jk}\ .
\end{equation*}
In particular, given a solution $v_{j,n-1}^{(0)}$, all the possible integer solutions for $v_{j,n-1}$ are given by
\begin{equation*}
    v_{j,n-1}=v_{j,n-1}^{(0)} - \frac{k_n}{k_{n-1}}\cdot h_{j,n-1} = v_{j,n-1}^{(0)} - v_{n-1,n-1}\cdot h_{j,n-1}\ ,\quad \forall\ h_{j,n-1}\in\Z\ .
\end{equation*}
Therefore, the division algorithm ensures the existence of a unique choice for $v_{j,n-1}$ such that
\begin{equation*}
    \forall\ i\leq j\leq n-1 \quad 0\leq v_{j,n-1}<v_{n-1,n-1}\ .
\end{equation*}
The same argument ensures the existence of a unique choice for $v_{j,n-1}$ such that
\begin{equation*}
    \forall\ 1\leq j\leq i-1 \quad 0\leq v_{j,n-1}<v_{n-1,n-1}\ .
\end{equation*}
Then the $(n-1)$--th column in $V^0$ can be uniquely chosen with non-negative entries satisfying the HNF condition. By completing the iteration, $V_0$ can then be uniquely chosen in HNF. Consequently $\v_0$ has to necessarily admits only negative entries. To prove that $V'$ is in HNF it suffices to observe that, for $V'$, the function $f:\{1,\ldots,n\}\rightarrow\{1,\ldots,n+1\}$, in Definition \ref{def:HNF}, is given by setting $f(i)=i$, for any $1\leq\ i\leq n$. Then $V'$ is in HNF if and only $V^0$ is in HNF, since there are no condition for the entries of $\v_0$ which is the $(n+1)$--th column of $V'$.
\end{proof}

\begin{remark}\label{rem:Q-canonico} When the weights vector $Q$ is fixed, a significant consequence of Proposition \ref{prop:fan minimale} is that the fan of $\P(Q)$ presented in (4) is unique and is given by the HNF of a matrix $V$ associated with any fan of $\P(Q)$. This gives rise to an algorithm producing the $Q$--canonical fan which has been implemented in \cite{RTpdf,RT-Maple}, \S 0.b.

\noindent Finally let us underline that the uniqueness of the $Q$--canonical fan of $\P(Q)$ depends on the weights order in $Q$, since the permutation group $\mathfrak{S}_{n+1}\subset\GL(n+1,\Z)$ acts on the right. Then we can't define a \emph{canonical} fan of $\P(Q)$ but just a $Q$--canonical one.
\end{remark}

\begin{example}\label{ex:Q-canonico} Let us apply the Proposition \ref{prop:fan minimale} to produce by hand the $Q$-canonical fan (hence a fan) of $\P(Q)$ for $Q=(2,3,4,15,25)$. First of all observe that in this case
\begin{equation*}
    k_1=\gcd(Q)=1\ ,\ k_2=d_1=1\ ,\ k_3=\gcd(2,15,25)=1\ ,\ k_4=\gcd(2,25)=1\ .
\end{equation*}
The matrix $V'$ in Proposition \ref{prop:fan minimale}(4) is in HNF, then it looks as follows
\begin{equation*}
    V'=\left(
         \begin{array}{ccccc}
           1 & 0 & 0 & v_{1,3} & v_{1,0} \\
           0 & 1 & 0 & v_{2,3} & v_{2,0} \\
           0 & 0 & 1 & v_{3,3} & v_{3,0} \\
           0 & 0 & 0 & 2 & v_{4,0} \\
         \end{array}
       \right)
\end{equation*}
with $0\leq v_{k,3}\leq 1$, for $1\leq k\leq 3$. Moreover we get the following conditions
\begin{eqnarray*}
  0 &=& q_0 v_{4,0}+\frac{q_0q_4}{k_4} = 2 v_{4,0} + 50\ \Rightarrow\ v_{4,0}=-25\\
  0 &=& q_0 v_{3,0}+\frac{k_3q_3}{k_4}+q_4 v_{3,3} = 2v_{3,0}+ 15 + 25 v_{3,3}\ \Rightarrow\ v_{3,3}=1\ \text{and}\ v_{3,0}=-20 \\
  0 &=& q_0 v_{2,0}+\frac{k_2q_2}{k_3}+q_4 v_{2,3} = 2 v_{2,0} + 4 + 25 v_{2,3}\ \Rightarrow\ v_{2,3}=0\ \text{and}\ v_{2,0}=-2 \\
  0 &=& q_0 v_{1,0}+\frac{k_1q_1}{k_2}+q_4 v_{1,3} = 2 v_{1,0} + 3 + 25 v_{1,3}\ \Rightarrow\ v_{1,3}=1\ \text{and}\ v_{1,0}=-14 \\
\end{eqnarray*}
giving the following $Q$--canonical fan for $\P(2,3,4,15,25)$\ :
\begin{equation*}
    \Si=\fan\left(\left(
                         \begin{array}{c}
                           -14 \\
                           -2 \\
                           -20 \\
                           -25 \\
                         \end{array}
                       \right), \left(
                                  \begin{array}{c}
                                    1 \\
                                    0 \\
                                    0 \\
                                    0 \\
                                  \end{array}
                                \right), \left(
                                                \begin{array}{c}
                                                  0 \\
                                                  1 \\
                                                  0 \\
                                                  0 \\
                                                \end{array}
                                              \right), \left(
                                                              \begin{array}{c}
                                                                0 \\
                                                                0 \\
                                                                1 \\
                                                                0 \\
                                                              \end{array}
                                                            \right), \left(
                                                                            \begin{array}{c}
                                                                              1 \\
                                                                              0 \\
                                                                              1 \\
                                                                              2 \\
                                                                            \end{array}
                                                                          \right)
    \right)\ .
\end{equation*}
\end{example}

\oneline
Let us conclude the present section with the following result, which will be useful later in (when proving Lemma \ref{lm:equivarianza} and Proposition \ref{prop:perm-equivarianza}).

\begin{proposition}\label{prop:induzione} Let $V=(\v_0,\ldots,\v_n)$ be the $Q$--canonical fan matrix with $Q=(q_0,\ldots,q_n)$, as constructed in the previous Proposition \ref{prop:fan minimale}. Then the matrix $\widehat{V}$, obtained from $V$ by multiplying $\v_0$ by $k_2$ and removing the column $\v_1$ and the first row, is the $\widehat{Q}$--canonical fan matrix, where $\widehat{Q}=\left(q_0/k_2, q_2,\ldots,q_n\right)$.
\end{proposition}

\begin{proof} $\widehat{V}$ is the following matrix
\begin{equation*}
    \widehat{V}=\left(
                    \begin{array}{cccccc}
                      k_2v_{2,0} & k_3/k_2 & v_{2,3} & \cdots& \cdots & v_{2,n} \\
                      k_2v_{3,0} & 0 & k_4/k_3 & \cdots &\cdots & v_{3,n} \\
                      \vdots & 0 & & \ddots &  & \vdots \\
                      \vdots & \vdots& & & k_n/k_{n-1}& v_{n-1,n}\\
                      k_2v_{n,0} & 0 & \cdots& 0 & 0 & q_0/k_n \\
                    \end{array}
                  \right)
\end{equation*}
Then $\widehat{V}_0=q_0/k_2$. On the other hand, by Proposition \ref{prop:fan minimale} we know that
\begin{eqnarray*}
    \forall\ 2\leq i\leq n-1\quad q_0v_{i,0}&=& -q_i\frac{k_{i+1}}{k_i}-\sum_{l=i+1}^n q_l v_{i,l}\\
                                  q_0v_{n,0}&=& -q_n\frac{q_0}{k_n}
\end{eqnarray*}
guaranteeing that $\widehat{V}$ satisfies the condition (4) of Theorem \ref{thm:fan} for the weights $\widehat{Q}$. Moreover $\widehat{V}$ is the $\widehat{Q}$--canonical fan matrix since $\widehat{V}^0$ is clearly still in HNF when $V^0$ is in HNF.
\end{proof}

\section{Characterization of polytopes giving $\P(Q)$}\label{sez: politopi}

The present section is devoted to answer questions (I) and (II) opening the previous section \ref{sez: fan}, in the case of polytopes. As recalled in \ref{sssez:politopi}, an integral polytope $\D$ corresponds to a \emph{polarized toric variety} $(\P_{\D},\mathcal{O}(1))$. Then those questions have to be reformulated, for polytopes, as follows:
\begin{itemize}
  \item[(A)] \emph{given a weights vector $Q$ and an integral polytope $\D$, when there exists a positive integer $m$ such that $(\P_{\D},\mathcal{O}(1))\cong (\P(Q),\mathcal{O}(m))$?}
  \item[(B)] \emph{given an integral polytope $\D$, when there exist a weights vector $Q$ and a positive integer $m$ such that $(\P_{\D},\mathcal{O}(1))\cong(\P(Q),\mathcal{O}(m))$?}
\end{itemize}

\subsection{From fans to polytopes and back}\label{fan_to_polytope}

 Consider the fan $\Si:=\fan(\v_0,\ldots,\v_n)$, generated by $n+1$ integer vectors satisfying the equivalent conditions (2), (3) and (4) in Theorem \ref{thm:fan}. Let $V=(\v_0,\ldots,\v_n)=(v_{ij})$ be the associated matrix. Let $V^0$ be the $n\times n$ sub-matrix of $V$ obtained by removing the first column: then $V^0=(\v_1,\ldots,\v_n)=(v_{ik})$ with $1\leq i\leq n$ and $1\leq k\leq n$.

\begin{definition}\label{def:Wtrasversa} Let $V\in \Mat(n,n+1;\Z)$ be a matrix whose maximal minors do not vanish i.e., in the same notation given above, $V_l\neq 0$ for every $0\leq l\leq n$. Consider the vector of absolute values of maximal minors $Q=(|V_0|,\ldots,|V_{n}|)$. Recalling \ref{trasversa}, the \emph{$(0,Q)$-weighted transverse matrix} of $V$ (or simply \emph{weighted transverse}) is defined to be the following $n\times n$ rational matrix
\begin{equation*}
    (V^0)^*_Q:=(V^0)^*\cdot (\d\ I^0_Q)
\end{equation*}
where $I^0_Q:=\diag(1/|V_1|,\ldots,1/|V_n|)$ and $\d:=\lcm(|V_0|,\ldots,|V_n|)$.
\end{definition}

\begin{remark} If $V\in\mathfrak{V}_n$ then the following Theorem \ref{thm:fan-politopi} implicitly shows that the weighted transverse matrix $(V^0)^*_Q$ is a $n\times n$ \emph{integral} matrix. In particular this fact will also be explicitly proved in Proposition \ref{prop:integrità}.
\end{remark}

The minimal (very) ample line bundle of the wps $\P(Q)=X(\Si)$ is given by $\mathcal{O}_{\P(Q)}(1)$ defined in (\ref{notazione}). If $D_j$ is the toric invariant divisor associated with $\v_j\in\Si(1)$ then $(\d'/q'_j) D_j$ is an ample divisor in the linear system $|\mathcal{O}_{\P(Q)}(1)|$, where as usual $Q'=(q'_0,\ldots,q'_n)$ is the reduced weights vector of $Q$ and $\d'=\lcm(Q')$. Set $\D_j$ be the integral polytope associated with the divisor $H=(\d'/q'_j) D_j$, like in (\ref{politopo}). One can easily check that $\D_j$ is the convex hull $\conv(\mathbf{0},\w_1,\ldots,\w_n)$ of the origin with $n$ points $\w_1,\ldots,\w_n\in M_{\R}$: in particular \emph{the ampleness of $(\d'/q'_j) D_j$ implies that $\{\w_1,\ldots,\w_n\}$ is a set of $n$ distinct, integral, non-zero vectors} (\cite{Oda} Corollary 2.14).

Let $\mathfrak{P}_n$ be the set of integral polytopes in $M_{\R}$ obtained as the convex hull of the origin with $n$ distinct, integral, non-zero vectors. Then we have established maps
\begin{equation}\label{fan-polytope}
    \begin{array}{cccc}
      \forall\ 0\leq j\leq n\ ,\quad\D^j_Q: & \mathfrak{F}(Q) & \longrightarrow & \mathfrak{P}_n \\
       & \Si & \mapsto & \D^j_Q(\Si):=\D_j
    \end{array}
\end{equation}
Let $W=(w_{ik})$ be the $n\times n$ matrix of the components of vectors $\w_1,\ldots,\w_n\in M_{\R}$ over the dual basis: namely
\begin{equation*}
    \forall\ k=1,\ldots,n\quad \w_k=\sum_{i=1}^n w_{ik}\e^{\vee}_i
\end{equation*}
where $\{\e^{\vee}_1,\ldots,\e^{\vee}_n\}$ is the dual basis of $\{\e_1,\ldots,\e_n\}$. Then we get the following representation of the map $\D^0_Q$:

\begin{theorem}\label{thm:fan-politopi} Given the fan $\Si:=\fan(\v_0,\ldots,\v_n)\in\mathfrak{F}(Q)$, the image $\D^0_Q(\Si)$ defined in (\ref{fan-polytope}) is the convex hull $\conv(\mathbf{0},\w_1,\ldots,\w_n)$ of the origin with the $n$ distinct, integral, non-zero vectors $\w_1,\ldots,\w_n\in M_{\R}$ giving the columns of the $(0,Q)$-weighted transverse matrix of $V=(\v_0,\ldots,\v_n)$, i.e.
\begin{equation*}
    W=(V^0)^*_Q\ ,
\end{equation*}
where $Q=(|V_0|,\ldots,|V_n|)$.
Namely the entries of $W$ are given by
\begin{equation*}
    \forall\ 1\leq i\leq n\ ,\ 1\leq k\leq n\quad w_{ik}=\frac{\d V^0_{ik}}{q_k V_0}
\end{equation*}
where $V^0_{ik}$ is the cofactor of $v_{ik}$ in $V^0$ and $V_0=\det(V^0)=\pm q_0$ (by either $(3)$ or $(4)$ in Theorem \ref{thm:fan}).
\end{theorem}

\begin{proof} Recalling (\ref{politopo}), to define $\D^0_Q(\Si)=\D_0$ one has to write down the hyperplanes of $M_{\R}$
\begin{equation}\label{iperpiani}
    \forall \rho\in\Si(1)\quad \langle\u,\n_{\rho}\rangle = - a_{\rho}\ ,\quad\text{where $\n_{\rho}$ generates $\rho\cap N$\ ,}
\end{equation}
for the divisor $H=(\d'/q'_0)D_0$. Since $\Si(1)=\{\langle\v_j\rangle\subset N_{\R} |\ j=0,\ldots,n\}$ the hyperplanes (\ref{iperpiani}) are then given by
\begin{eqnarray}\label{iperpiani 2}
  \sum_{i=1}^n n_{i0} u_i &=& - \d'/q'_0\\
\nonumber
  \forall\ k=1,\ldots,n\quad\sum_{i=1}^n n_{ik} u_i &=& 0
\end{eqnarray}
where $\n_j=\sum_{i=1}^n n_{ij}\e_i$ generates the 1-dimensional cone $\langle\v_j\rangle\cap N$. In proving the implication $(2)\Rightarrow (1)$ of Theorem \ref{thm:fan}, it has been observed that $q'_{0}\n_{0}=-\sum_{k=1}^{n}q'_{k}\n_{k}$ (condition (2$'$)). Then the first equation in (\ref{iperpiani 2}) can be rewritten as follows
\begin{equation*}
    \sum_{i=1}^n \left(\sum_{k=1}^{n}q'_{k}n_{ik}\right) u_i = \d'\ .
\end{equation*}
Let us represent equations in (\ref{iperpiani 2}) by the following $(n+1)\times(n+1)$-matrix
\begin{equation*}
    M=\left(
        \begin{array}{ccccc}
          \sum_{k=1}^{n}q'_{k}n_{1k} & \cdots & \sum_{k=1}^{n}q'_{k}n_{nk} & \vline & \d' \\
          n_{11} & \cdots & n_{n1} & \vline & 0 \\
           & \vdots &  & \vline & \vdots \\
          n_{1n} & \cdots & n_{nn} & \vline & 0 \\
        \end{array}
      \right)\ .
\end{equation*}
For $j=0,1,\ldots,n$, the vertex $\w_j$ of $\D^0_Q(\Si)$ is then given by the (unique, for (3) in Theorem \ref{thm:fan} and recalling that $v_{ij}=d_jn_{ij}$) solution of the linear system associated with the matrix $M^{j+1}$, obtained removing the $(j+1)$-th row in $M$. Clearly $\w_0=0$. For $j=k=1,\ldots,n$ we get
\begin{equation*}
    w_{ik}=M_{k+1,i}/M_{k+1,n+1}
\end{equation*}
where $M_{a,b}$ is the $(a,b)$-cofactor in $M$. Observe that $M_{k+1,n+1}=(-1)^{k-1}q'_k\left.V_0\right/a_0$ and $M_{k+1,i}=(-1)^{k+1}\d' d_k \left.V^0_{ik}\right/a_0$. Then
\begin{equation*}
    w_{ik}=\frac{\d' d_k}{q'_k}\frac{V^0_{ik}}{V_0} = \frac{\d'a_kd_k}{q_k}v_{ik}^* = \frac{\d}{q_k}v_{ik}^*
\end{equation*}
where $v_{ik}^*=V^0_{ik}/V_0$ is the $(i,k)$-entry of $V^{0*}:=((V^0)^{-1})^T$. The last equality on the right is obtained by recalling Proposition \ref{prop:weight-rel.s}(5) and Proposition \ref{prop:lcm}.
\end{proof}

\begin{remark} Clearly same conclusions as in Theorem \ref{thm:fan-politopi} can be obtained by exchanging $0$ with any other value $j$ such that $0\leq j\leq n$.
\end{remark}

\begin{remark}\label{rem:wtrans-ridotta} Let $Q$ be  a weights vector whose reduction is given by $Q'$. Consider $\Si=\fan(\v_0,\ldots,\v_n)\in\mathfrak{F}(Q)$ and, for any $0\leq j\leq n$, consider the generator $\n_j$ of the semigroup $\langle\v_j\rangle\cap N$, where $N$ is the lattice generated by $\v_0,\ldots,\v_n$. Then Lemma \ref{lm:condizioni}(c) and Theorem \ref{thm:fan} ensure that $\Si:=\fan(\n_0,\ldots,\n_n)\in\mathfrak{F}(Q')$. Then the previous Theorem \ref{thm:fan-politopi} gives that
\begin{equation*}
    \D^0_Q(\Si) = \D^0_{Q'}(\Si)
\end{equation*}
since, recalling once again Propositions \ref{prop:weight-rel.s} and \ref{prop:lcm},  $$w_{ik}=(\d/q_k)(V^0_{ik}/V_0)=(\d'a/q'_ka_k)(N^0_{ik}/d_kN_0)= (\d'/q'_k)(N^0_{ik}/N_0)$$
(here $N$ denotes the matrix $N=(\n_0,\ldots,\n_n)$).
\end{remark}

\begin{example}\label{ex:Qpolitopo}
Let us still consider the Example \ref{ex:Q-canonico} to apply the weighted transversion and Theorem \ref{thm:fan-politopi} for producing by hand a polytope of a given wps $\P(Q)$ with the minimal polarization.

\noindent Recall that $Q=(2,3,4,15,25)$ and the matrix fan obtained in the Example \ref{ex:Q-canonico} is
\begin{equation*}
    V=\left(
       \begin{array}{ccccc}
         -14 & 1 & 0 & 0 & 1 \\
         -2 & 0 & 1 & 0 & 0 \\
         -20 & 0 & 0 & 1 & 1 \\
         -25 & 0 & 0 & 0 & 2 \\
       \end{array}
     \right)\quad\Longrightarrow\quad (V^0)^*=\frac{1}{2}\left(
                                      \begin{array}{cccc}
                                        2 & 0 & 0 & 0 \\
                                        0 & 2 & 0 & 0 \\
                                        0 & 0 & 2 & 0 \\
                                        -1 & 0 & -1 & 1 \\
                                      \end{array}
                                    \right)
     \ .
\end{equation*}
Since $\d=\lcm(2,3,4,15,25)=300$, we get
\begin{equation*}
    W=(V^0)^*_Q=(V^0)^*\cdot\d I_Q=150 \left(
                                      \begin{array}{cccc}
                                        2 & 0 & 0 & 0 \\
                                        0 & 2 & 0 & 0 \\
                                        0 & 0 & 2 & 0 \\
                                        -1 & 0 & -1 & 1 \\
                                      \end{array}
                                    \right)\cdot\left(
                                                  \begin{array}{cccc}
                                                    1/3 & 0 & 0 & 0 \\
                                                    0 & 1/4 & 0 & 0 \\
                                                    0 & 0 & 1/15 & 0 \\
                                                    0 & 0 & 0 & 1/25 \\
                                                  \end{array}
                                                \right)
\end{equation*}
giving $W=\left(
            \begin{array}{cccc}
              100 & 0 & 0 & 0 \\
              0 & 75 & 0 & 0 \\
              0 & 0 & 20 & 0 \\
              -50 & 0 & -10 & 6 \\
            \end{array}
          \right)
$. Then the polytope we are looking for is
\begin{equation*}
    \D=\conv\left(\left(
                    \begin{array}{c}
                      0 \\
                      0 \\
                      0 \\
                      0 \\
                    \end{array}
                  \right),\left(
                    \begin{array}{c}
                      100 \\
                      0 \\
                      0 \\
                      -50 \\
                    \end{array}
                  \right),\left(
                    \begin{array}{c}
                      0 \\
                      75 \\
                      0 \\
                      0 \\
                    \end{array}
                  \right),\left(
                    \begin{array}{c}
                      0 \\
                      0 \\
                      20 \\
                      -10 \\
                    \end{array}
                  \right),\left(
                    \begin{array}{c}
                      0 \\
                      0 \\
                      0 \\
                      6 \\
                    \end{array}
                  \right)
    \right)\ .
\end{equation*}
More precisely $(\P_{\D},\mathcal{O}(1))\cong(\P(Q),\d/q_0\ D_0)=(\P^4(2,3,4,15,25), 150\ D_0)$.
\end{example}

Recalling Algorithm \ref{alg:daQaF}, what has been observed in the previous Example \ref{ex:Qpolitopo} can be resumed by the following

\begin{algorithm}[Producing a polytope of $\P(Q)$ with a minimal polarization]\label{alg:daQaP} See 0.c.3 in \cite{RTpdf,RT-Maple} for a Maple implementation.
\begin{itemize}
\item Input: the weights vector $Q=(q_0,\ldots,q_n)$.
\item Apply Algorithm \ref{alg:daQaF} above in order to compute a fan $F:=({\bf v}_0,...,{\bf v}_n)$ associated with $Q$;
\item Put $V^0:=Mat({\bf v}_1,...,{\bf v}_n)$ and compute the weighted transverse $W:=(V^0)^*_Q$ (Definition \ref{def:Wtrasversa}).
\item Define ${\mathcal P}$ to be the set of points in $\mathbb{R}^n$ consisting of the origin and the columns of $W$.
\item Output ${\mathcal P}$.
\end{itemize}
\end{algorithm}

\begin{definition}[$P$--admissible matrices]\label{def:P-admissible}  A square matrix $W\in\Mat(n,\Z)$ is called $P$--\emph{admissible} if there exist an $F$-admissible matrix $V\in \mathfrak{V}_n$ such that $W$ is the weighted transverse matrix of $V$, which is
$$W=(V^0)_Q^*\quad\text{with}\quad Q=\left(|V_0|,\ldots,|V_n|\right)\ .$$
In other terms $W=(\w_1,\ldots,\w_n)$ is admissible if and only if the polytope $$\conv(\mathbf{0},\w_1,\ldots,\w_n)$$ belongs to the image of a map $\D^0_Q$, as defined in (\ref{fan-polytope}).
In this case we say that $Q,W,V$ are \emph{associated} to each other.

\noindent Let us denote $\mathfrak{W}_n\subset\GL(n,\Q)\cap\Mat(n,\Z)$ the subset of $P$--admissible matrices: notice that any such matrix has integer entries by either Theorem \ref{thm:fan-politopi} or the following Proposition \ref{prop:integrità}.
\end{definition}

\begin{remark}
Remark \ref{rem:wtrans-ridotta} guarantees that weights vectors $Q_1$ and $Q_2$ admitting the same reduction $Q'$ are associated with the same $P$--admissible matrix $W$, which is the $P$--admissible matrix associated with the reduced weights vector $Q'$. What is not a priori clear is, viceversa, guaranteeing that there exists a unique reduced weights vector $Q'$ to which $W$ is associated. This fact will follow by Proposition \ref{prop:admuno}(c).
\end{remark}

\begin{definition}\label{def:what} Consider a matrix $W\in\GL(n,\Q)\cap\Mat(n,\Z)$. Recall that the \emph{adjoint matrix of} $W$ is defined by setting $Adj(W):=\det(W)\ W^{-1}$. Let $s_i$ be the $\gcd$ of entries in the $i$-th row of $Adj(W)$. Then let
us then define the \emph{what} of $W$ as follows
\begin{eqnarray*}
    \widehat{W}&:=& \frac{|det(W)|}{det(W)}\diag\left({1\over s_1},\ldots,{1\over s_n}\right)\cdot Adj(W)\\
    &=& \diag\left({|det(W)|\over s_1},\ldots,{|det(W)|\over s_n}\right)\cdot W^{-1}
\end{eqnarray*}
\end{definition}

Notice that if $V$ is a square matrix in $\Mat(n,\mathbb{Z})$ such that $V\cdot W$ is a diagonal matrix with positive entries then
\begin{equation}\label{diagonale}
    V=\diag(r_1,\ldots,r_n)\cdot\widehat{W}
\end{equation}
for some $r_1,\ldots,r_n\in \N$.

\begin{proposition}\label{prop:admuno}
Let $W$ be a $P$--admissible matrix and let $Q=(q_0,\ldots,q_n)$ be a \emph{reduced} weights vector associated to $W$. Then
\begin{itemize}
\item[(a)] $\left(\widehat W^T\right)^*_Q=W$;
\item[(b)] if $s:=\gcd(s_1,\ldots,s_n)$ is the greatest common divisor of the terms in $Adj(W)$ then
\begin{equation*}
    q_0=|det(\widehat W)|\quad,\quad \forall\ 1\leq i\leq n\quad q_i=\frac {s_i} s\quad,\quad\lcm(Q)=\frac{|\det(W)|}{s}
\end{equation*}
\item[(c)] if $Q_1$ and $Q_2$ are reduced weights vectors associated with the same $P$--admissible matrix $W$, then $Q_1=Q_2$;
\item[(d)] there exists a unique $F$--admissible matrix $V$ associated with $W$ and $Q$ i.e. such that $W=\left(V^0\right)^*_Q$ with $Q=\left(|V_0|,\ldots,|V_n|\right)$.
    \end{itemize}
    \end{proposition}
    \begin{proof} (a). $W$ is a $P$--admissible matrix. Then there exists a $F$--admissible matrix $V$ such that $W=((V^0)^T)^{-1}\d I_Q$ and $Q=(|V_0|,\ldots,|V_n|)$, meaning that $(V^0)^T W=\delta I^0_Q$ is diagonal with positive entries. Recalling (\ref{diagonale}) we get that $(V^0)^T=\diag(r_1,\ldots,r_n)\cdot\widehat{W}$ for some $r_1,\ldots,r_n\in \N$. But $Q$ is reduced, which implies that the columns of $V^0$ have coprime entries. Therefore $r_1=\cdots=r_n=1$ and $(V^0)^T=\widehat W$. (a) follows immediately.\par
    \noindent (b) On the one hand $\widehat{W}\cdot W=\diag\left(|\det W|/s_1,\ldots,|\det W|/s_n\right)$. On the other hand, by (a), $\widehat{W}=(V^0)^T$ and $\widehat{W}\cdot W=\diag\left(\d/q_1,\ldots,\d/q_n\right)$, where $\d:=\lcm(Q)$. Therefore
    \begin{equation}\label{uguaglianza delta}
        \forall\ 1\leq i\leq n\quad{\d\over q_i}={|\det W|\over s_i}\ .
    \end{equation}
    Observe now that
    \begin{eqnarray*}
        \lcm\left({\d\over q_1},\ldots,{\d\over q_n}\right)&=&\frac{\d}{\gcd(q_1,\ldots,q_n)}=\d\\
        \lcm\left({|\det W|\over s_1},\ldots,{|\det W|\over s_n}\right)&=&\frac{|\det W|}{s}
    \end{eqnarray*}
    Then (\ref{uguaglianza delta}) gives that $\d=|\det W|/s$ and, for any $1\leq i\leq n$, $q_i=s_i/ s$.
    Finally (a) gives that $q_0=|V_0|=|det(\widehat W)|$.\par
    \noindent (c) follows immediately by the previous point (b).\par
    \noindent (d). If there exist two $F$--admissible matrix $U,V$ such that they are both associated with $W$ and $Q$, then
\begin{equation*}
    (\v_1,\ldots,\v_n)=V^0=U^0=(\u_1,\ldots,\u_n)\ \Rightarrow\ \v_0 =-\frac{1}{q_0}\sum_{i=1}^n q_i\v_i =-\frac{1}{q_0}\sum_{i=1}^n q_i\u_i=\u_0
\end{equation*}
implying that $V=U$.
\end{proof}

\begin{remark}\label{rem:inversione} In a sense the previous Proposition \ref{prop:admuno} states that, when restricted to wps fans associated with \emph{reduced} weights vector, the weighted transversion process giving a polytope starting from a fan, can be inverted by considering the \emph{transposed what} of the polytope matrix. Namely if $W$ is a  polytope matrix of $(\P(Q),\mathcal{O}(1))$, with $Q$ reduced, then $V:=\left(
       \begin{array}{ccc}
         \v_0 & \vline & \widehat{W}^T \\
       \end{array}
     \right)$ is a fan matrix of $\P(Q)$ when $\v_0$ is defined by setting $\v_0=-(\sum_{i=1}^n q_i\v_i)/q_0$, where $(\v_1,\ldots,\v_n)=\widehat{W}^T$. At this purpose see also the following Propositions \ref{prop:suriettività} and \ref{prop:classification_Poly}. The following Proposition \ref{prop:poli-condizioni} ensures that $V$ is a well defined matrix with integer entries if and only if $W$ is a polytope matrix of $(\P(Q),\mathcal{O}(1))$, for some reduced weights vector $Q$. This correspondence between fans and polytopes of a wps gives rise to easy and fast procedures, relating each other the toric data of this particular complete toric variety, which has been implemented in \cite{RTpdf,RT-Maple}, \S 0.c and \S 2.4.
\end{remark}

    \begin{proposition}\label{prop:poli-condizioni} Let $W=(w_{ij})\in\GL(n,\Q)\cap\Mat(n,\Z)$ be a matrix such that $\gcd(w_{ij})=1$. Let  $s$ be the greatest common divisor of the entries in $Adj(W)$ and ${\v}$ be the sum of the rows of $Adj(W)$. Define $q_0=|det(\widehat W)|$, $\delta=\frac{|\det W|} s$.
    The following statements are equivalent:
    \begin{itemize}
    \item[(a)] $W$ is $P$--admissible;
    \item[(b)] the vector $\v$ is divisible by $q_0s$;
        \item[(c)]  $q_0$ divides $\delta$ and the vector $\frac \delta {q_0}(1,\ldots,1)$ is in the lattice generated by the rows of $W$.
            \end{itemize}
            \end{proposition}
\begin{proof}
(a) $\Rightarrow$(b): If $W$ is $P$--admissible then there exist a unique reduced weights vector $Q$ and a unique $F$--admissible matrix $V$, associated with $W$ like in Definition \ref{def:P-admissible}. By Proposition \ref{prop:admuno}, $\widehat W=(V^0)^T$ and $Q=(q_0,\ldots,q_n)$ with $q_i= s_i/ s$, for $i=1,\ldots,n$. Let ${\v_i}$ be the i-th row of $\widehat W$; then $\sum_{i=1}^n q_i {\v_i}$ is divisible by $q_0$ since $\widehat W^T=V^0$ is $F$--admissible, meaning that its columns satisfy the relation $\sum_{i=1}^n q_i {\v_i}=-q_0\v_0$. Then $\sum_{i=1}^n s_i {\v_i}$ is divisible by $q_0s$ and $ s_i {\v_i}$ is the i-th row of $Adj(W)$.\par
\noindent (b) $\Rightarrow$(a): Assume that $q_0s$ divides any entry in $\v$. For $1\leq i\leq n$, let ${\v_i}$ be the i-th row of $\widehat W$ and $q_i=s_i/s$ be defined as in Proposition \ref{prop:admuno}(b); then $\sum_{i=1}^n q_i {\v_i}$ is divisible by $q_0$. Put $\v_0=-\frac 1{q_0}\sum_{i=1}^n q_i {\v_i}$. Then the matrix $$V:=\left(
       \begin{array}{ccc}
         \v_0 & \vline & \widehat{W}^T \\
       \end{array}
     \right)=
\left(\v_0,\v_1,\ldots,\v_n\right)$$
turns out to be $F$--admissible with respect to $Q$ by Theorem \ref{thm:fan}(4). Then $W=\left(V^0\right)^*_Q$ is $P$--admissible.\par
\noindent (b) $\Leftrightarrow$(c): the sum of the rows of $Adj(W)$ is the row vector $(1,\ldots,1)\cdot Adj(W)=(1,\ldots,1)\cdot det(W)W^{-1}$. Thus it is divisible by $q_0s$ if and only if there exists $(x_1,\ldots,x_n)\in\Z^n$ such that $(x_1,\ldots,x_n)\cdot W=\frac {\delta}{q_0}(1,\ldots,1)$, that is if and only if (c) holds.
\end{proof}

\begin{proposition}\label{prop:integrità}
If $V=(\v_0,\v_1,\ldots,\v_n)$ is a fan matrix of $\P(Q)$, with $Q=(q_0,...,q_n)$, then the weighted transverse $(V^0)^*_Q$ has integral entries.
\end{proposition}

\begin{proof}
Set $W=Adj(V^0)$ and let ${\w}_i$ be the $i$-th row of $W$, for $i=1,\dots,n$. Observe that Definition \ref{def:what} and Theorem \ref{thm:fan} give
\begin{eqnarray*}
  |{\w}_i\cdot {\v}_i| &=& q_0 \\
  {\w}_i\cdot {\v}_k &=& 0 \quad\text{for $1\leq k\leq n$ and $k\neq i$}\\
  |{\w}_i\cdot {\v}_0| &=& q_i \ .
\end{eqnarray*}
Therefore $\frac \delta {q_0q_i} {\w}_i\cdot {\v}_j\in \mathbb{Z}$ for any $0\leq j\leq n$. This means that $$\forall\ 1\leq i\leq n\quad \frac \delta {q_0q_i} {\w}_i\in \mathbb{Z}^n$$
since ${\mathcal L}( {\bf v}_0,..., {\v}_n)=\mathbb{Z}^n$. The proof ends up by transposing $W$.
\end{proof}

\subsection{Characterizing the polytope of a polarized wps}

The previous Proposition \ref{prop:poli-condizioni} gives an answer to questions (A) and (B) opening the present section. In fact, let us first of all observe that, given an integral polytope $\D$, up to an integral translation, we can assume the origin 0 to be a vertex of $\D$ and write $\D=\conv(\mathbf{0},\w_1,\ldots,\w_n)$, for a suitable subset $\{\w_1,\ldots,\w_n\}\subset M$. Let $W:=(\w_1,\ldots,\w_n)$ be the associated polytope matrix. Then the following result is a consequence of Propositions \ref{prop:admuno} and \ref{prop:poli-condizioni} answering question (A):

\begin{theorem}\label{thm:(A)} Let $\D=\conv(\mathbf{0},\w_1,\ldots,\w_n)\subset M_{\R}$ be a $n$-dimensional integral polytope and $Q=(q_0,\ldots,q_n)$ be a weights vector. Set $m:=\gcd(w_{ij})$ and define $W':={1\over m}W$. Then the following facts are equivalent:
\begin{enumerate}
  \item $(\P_{\D},\mathcal{O}(1))\cong (\P(Q),\mathcal{O}(m))$,
  \item $W'$ is a $P$--admissible matrix associated with $Q$,
  \item $\widehat{W'}\cdot W'= \d' I_{Q'}$, where $Q'$ is the reduced weights vector of $Q$ and $\d':=\lcm(Q')$.
\end{enumerate}
\end{theorem}

\begin{proof} (1)$\Rightarrow$(2): There exists a divisor $D$ of $\P(Q)$, belonging to the linear system $|\mathcal{O}(m)|$, such that $\D=\D_{D}$. Moreover there exists a divisor $D'\in|\mathcal{O}(1)|$ such that $D=mD'$ and in particular $\D=\D_{D}=m\D_{D'}$. This means that $r=n$ since $D'$ is ample, $\D_{D'}=\conv(\mathbf{0},\w'_1,\ldots,\w'_n)$ and $W':=(\w'_1,\ldots,\w'_n)={1 \over m}W$ is a $P$--admissible matrix associated with $Q$.

(2)$\Rightarrow$(3): By Propositions \ref{prop:admuno} and \ref{prop:poli-condizioni}, the matrix
\begin{equation}\label{V}
V:=\left(
       \begin{array}{ccc}
         \v_0 & \vline & \widehat{W'}^T \\
       \end{array}
     \right)=
\left(\v_0,\v_1,\ldots,\v_n\right)
\end{equation}
where $\v_0=-(\sum_{i=1}^n q'_i {\v_i})/q'_0$, is a $F$--admissible matrix with respect to the reduction $Q'$ of $Q$, meaning that
\begin{equation*}
    W'=\left(\widehat{W'}^T\right)^*_{Q'}= (\widehat{W'})^{-1}(\d' I_{Q'})\ .
\end{equation*}

(3)$\Rightarrow$(1): Since $\widehat{W'}\cdot W'= \d' I_{Q'}$, the matrix $V$ defined in (\ref{V}) is $F$--admissible with respect to $Q'$. Then $W'$ is the polytope matrix of $\D_{D'}$ for some divisor $D'\in\mathcal{O}(1)$ of $\P(Q')\cong\P(Q)$. Moreover $W=mW'$ is the polytope matrix of $\D:=\D_{mD'}$ giving (1).
\end{proof}

On the other hand the following further consequence of Propositions \ref{prop:admuno} and \ref{prop:poli-condizioni}  answers question (B):

\begin{theorem}\label{thm:(B)} Let $\D=\conv(\mathbf{0},\w_1,\ldots,\w_n)\subset M_{\R}$ be a $n$-dimensional integral polytope and $W=(\w_1,\ldots,\w_n)$ be the associated polytope matrix. Set $m:=\gcd(w_{ij})$ and define $W':={1\over m}W$. Then the following facts are equivalent:
\begin{enumerate}
    \item there exists a weights vector $Q=(q_0,\ldots,q_n)$ such that $(\P_{\D},\mathcal{O}(1))\cong (\P(Q),\mathcal{O}(m))$,
    \item $W'$ satisfies one of the equivalent conditions (a), (b), (c) in Proposition \ref{prop:poli-condizioni}.
\end{enumerate}
\end{theorem}

\begin{proof} Define $q_i:=s_i/ s$ where $s_i$ is the greatest common divisor of entries in the $i$--th row of $Adj(W')$ and $s$ is the greatest common divisor of entries in $Adj(W')$: then we get a reduced weights vector $Q=(q_0,q_1,\ldots,q_n)$. The equivalent conditions of Proposition \ref{prop:poli-condizioni} mean that $W'={1 \over m}W$ is a $P$--admissible matrix with respect to $Q$. The previous Theorem \ref{thm:(A)} ends up the proof.
\end{proof}

These results give rise to the following algorithm for \emph{recognizing a polarized wps polytope}:

\begin{algorithm}[Recognizing a polytope of a polarized $\P(Q)$]\label{alg:politopo} See \S 2 in \cite{RTpdf,RT-Maple} for a Maple implementation.
\begin{itemize}
\item Input: a polytope ${\mathcal P}:=(P_0,...,P_n)$.
\item Construct the matrix $W\in\Mat(n,\mathbb{Z})$ whose $i$-th column is given by coordinates of the point $P_i-P_0$.
\item Compute $m:=\gcd(W)$ and the normalized matrix $W':=W/m$, like in Theorems \ref{thm:(A)} and \ref{thm:(B)}.
\item Compute the adjoint matrix $Adj(W')$ of $W'$.
\item Define $s$ and $s_i$, for $i=1,\ldots,n$, like in Proposition \ref{prop:admuno} and in Definition \ref{def:what}, with respect to the normalized matrix $W'$.
\item Set $\forall\ 1\leq i\leq n\quad q_i:=\frac {s_i} s$.
\item Consider the matrix $\widehat{W'}$, as defined in Definition \ref{def:what}.
\item Set $q_0:=\det(\widehat{W'})$.
\item For $i=1,...,n$ let $\v_i$ be the $i$-th row of $\widehat{W'}$ and define $\v_0=-\frac 1{q_0} \sum_{i=1}^n q_i\v_i$.
\item If $\v_0\in \mathbb{Z}^n$, then ${\mathcal P}$ is associated to the polarized wps whose weights vector is given by $Q:=(q_0,...,q_n)$ and whose polarization is given by $m$; else ${\mathcal P}$ is not associated to a $WPS$.
\item Output: either the weights vector $Q$ and the polarization $m$ or an error message.
\end{itemize}
\end{algorithm}

\subsection{Equivalence of polytopes} Recalling Definition \ref{def:Wtrasversa}, let us define the \emph{weighted transverse map}
\begin{equation}\label{Wtrasversa map}
\begin{array}{ccccc}
  \tau: & \mathfrak{V}_n & \longrightarrow & \mathfrak{W}_n &\\
   & V & \longmapsto & (V^0)^*_Q & \text{with $Q:=\left(|V_0|,\ldots,|V_n|\right)$}\ .
\end{array}
\end{equation}

\begin{proposition}\label{prop:suriettività} The weighted transverse map $\tau$ is surjective and the fibers of $\tau$ are as follows:
    \begin{equation*}
        \forall\ W\in\mathfrak{W}_n\quad\tau^{-1}(W)=\{V\in\mathfrak{V}_n\ |\ \text{$(|V_0|,\ldots,|V_n|)$ reduces to $Q=Q(W)$}\},
    \end{equation*}
    where $Q(W)=\left(|\det(\widehat{W})|, s_1/s,\ldots,s_n/s\right)$.
\end{proposition}

\begin{proof} The map $\tau$ is surjective by the Definition \ref{def:P-admissible} of $P$--admissible matrix and Remark \ref{rem:wtrans-ridotta}.
\end{proof}

Recalling that the permutation group $\mathfrak{S}_{n+1}\subset\GL(n+1,\Z)$ acts on $\mathfrak{V}_n$ by right multiplication, giving the quotient (\ref{Fn}), we are now going to define a (right) action of $\mathfrak{S}_{n+1}$ over $\mathfrak{W}_n$ in such a way that the map $\tau$ turns out to be equivariant. Namely, proceed as follows:
\begin{itemize}
  \item[-] embed $\mathfrak{W}_n$ in $\GL(n+1,\Q)\cap\Mat(n+1,\Z)$ by setting
\begin{equation*}
    W=(\w_1,\ldots,\w_n)\longmapsto\widetilde{W}:=\left(
                                     \begin{array}{cccc}
                                       0   &  &  \\
                                       \vdots &  & W &  \\
                                       0 &  &  &  \\
                                       1 & 1 & \cdots & 1 \\
                                     \end{array}
                                   \right)= \left(
                                              \begin{array}{cccc}
                                                \mathbf{0} & \w_1 & \cdots & \w_n \\
                                                1 & 1 & \cdots & 1 \\
                                              \end{array}
                                            \right)
\ ,
\end{equation*}

  \item[-] act now by right multiplication obtaining
$$\sigma\left(\widetilde{W}\right):=\widetilde{W}\cdot A(\s)=\left(
                                         \begin{array}{cccc}
                                         \w_{\s(0)} & \w_{\s(1)} & \cdots & \w_{\s(n)} \\
                                          1 & 1 & \cdots & 1 \\
                                                        \end{array}
                                                      \right)
$$
where $\w_0:=\mathbf{0}$ and $A(\s)\in\GL(n+1,\Z)$ is the matrix associated with the permutation $\s\in\mathfrak{S}_{n+1}$,
  \item[-] multiply on the left by the affine matrix of $\GL(n+1,\Z)$ translating $\w_{\s(0)}$ in the origin, which is $T(W,\s):=\left(
            \begin{array}{cc}
               I_n &  -\w_{\s(0)} \\
               \mathbf{0}&  1 \\
            \end{array}
          \right)
$, giving
\begin{eqnarray*}
    T(W,\s)\cdot\widetilde{W}\cdot A(\s)&=&\left(
                                     \begin{array}{cccc}
                                       0   &  &  \\
                                       \vdots &  & W' &  \\
                                       0 &  &  &  \\
                                       1 & 1 & \cdots & 1 \\
                                     \end{array}
                                   \right)\\
                                        &=&\left(
                                             \begin{array}{cccc}
                                               \mathbf{0} & \w_{\s(1)}-\w_{\s(0)} & \cdots & \w_{\s(n)}-\w_{\s(0)} \\
                                               1 & 1 & \cdots & 1 \\
                                             \end{array}
                                           \right)
\ ,
\end{eqnarray*}
    \item[-] restrict to consider the submatrix $W'=\left(\w_{\s(1)}-\w_{\s(0)}, \ldots ,\w_{\s(n)}-\w_{\s(0)}\right)$ and set $W\ast \s:=W'$.
\end{itemize}
$W\ast \s$ is a $P$--admissible matrix since
$$\conv\left(\mathbf{0},\w_{\s(1)}-\w_{\s(0)}, \ldots ,\w_{\s(n)}-\w_{\s(0)}\right)=\D^0_{\s(Q)}\left(\fan(\v_{\s(0)},\ldots,\v_{\s(n)})\right)\ ,$$
where $W=\tau(V)$ and $V=(\v_0,\ldots,\v_n)$.
Then it is well defined an action of $\mathfrak{S}_{n+1}$ over $\mathfrak{W}_n$.

\begin{proposition}\label{prop:perm-equivarianza} The weighted transverse map $\tau$ defined in (\ref{Wtrasversa map}) is equivariant with respect to the left action of $\GL(n,\Z)$ and the right action of the permutation group $\mathfrak{S}_{n+1}$, namely
\begin{eqnarray*}
    \forall\ A\in\GL(n,\Z)\ ,\ \forall\ V\in\mathfrak{V}_n\quad \tau(A\cdot V) &=& A^*\cdot\tau(V) \\
    \forall\ \s\in\mathfrak{S}_{n+1}\ ,\ \forall\ V\in\mathfrak{V}_n\quad \tau(\sigma(V))&=& \tau(V)\ast\s\ .
\end{eqnarray*}
\end{proposition}

\begin{proof} For the $\GL(n,\Z)$ left action, given $A\in\GL(n,\Z)$, one has
\begin{equation*}
    \tau(A\cdot V)=\tau\left(\left(A\cdot\v_0\ |\ A\cdot V^0\right)\right)=\left(A\cdot V^0\right)^*\cdot\left(\d I_{Q}\right) = A^*\cdot\left(V^0\right)^*\cdot\left(\d I_{Q}\right) = A^*\cdot \tau(V)\ ,
\end{equation*}
where $Q=\left(|V_0|,\ldots,|V_n|\right)$, as usual.

\noindent On the other hand, for the $\mathfrak{S}_{n+1}$ right action one has to prove that
\begin{equation}\label{perm_equivarianza}
    \forall\ \s\in\mathfrak{S}_n+1\quad \tau(V\cdot A(\s)) = \tau(V)\ast\s\ ,
\end{equation}
where $A(\s)\in\GL(n+1,\Z)$ is the unimodular matrix naturally associated with the permutation $\s$.
We can distinguish two cases: either $\s(0)=0$ or $\s(0)=s$ with $1\leq s\leq n$. In the first case the matrix $A(\s)$ assumes the following shape
$$A(\s)=\left(
                                             \begin{array}{cc}
                                               1 & \mathbf{0} \\
                                               \mathbf{0} & A'(\s) \\
                                             \end{array}
                                           \right)\in\mathfrak{S}_{n+1}\subset\GL(n+1,\Z)\ ,
$$
with $A'(\s)\in\mathfrak{S}_n\subset\GL(n,\Z)$. Then, given $V=(\v_0,\ldots,\v_n)\in\mathfrak{V}_n$, the first term in (\ref{perm_equivarianza}) rewrites as follows
\begin{eqnarray*}
    \tau(V\cdot A(\s))&=&\tau\left(\left(\v_0\ |\ V^0\cdot A'(\s)\right)\right)\\
                                  &=&\left(V^0\cdot A'(\s)\right)^*_{\s(Q)}= \left(V^0\right)^*\cdot\left(A'\right)^*\cdot\left(\d I_{\s(Q)}\right)\ .
\end{eqnarray*}
Notice that $A'(\s)$ is an orthogonal matrix, giving $\left(A'\right)^*=A'$. Moreover $I_{\s(Q)}= A'(\s)^T\cdot I_Q\cdot A'(\s)$. Then
\begin{equation*}
    \tau(V\cdot A(\s))= \left(V^0\right)^*\cdot A'\cdot \left(A'\right)^T\cdot\left(\d I_{Q}\right)\cdot A' = \tau(V)\cdot A'(\s) =\tau(V)\ast\s\ ,
\end{equation*}
where the last equality is obtained by observing that $\w_{\s(0)}=\w_0=\mathbf{0}$.

\noindent In the second case, when $\s(0)=s\neq 0$, the first term in (\ref{perm_equivarianza}) rewrites as follows
\begin{equation*}
    \tau(V\cdot A(\s))= \tau\left((\v_{\s(0)},\ldots,\v_{\s(n)})\right)=\left(\v_{\s(1)},\ldots,\v_{\s(n)}\right)^*_{\s(Q)}=(w'_{ik})
\end{equation*}
where, recalling Theorem \ref{thm:fan-politopi},
$$w'_{ik}=\frac{\d V^{\s(0)}_{i,\s(k)}}{q_{\s(k)}V_{\s(0)}}\ .$$
On the other hand, the second term in (\ref{perm_equivarianza}) is given by
\begin{equation*}
    \tau(V)\ast\s=\left(\w_{\s(1)}-\w_{\s(0)},\ldots,\w_{\s(n)}-\w_{\s(0)}\right)=(w''_{ik})
\end{equation*}
with
\begin{equation*}
    w''_{ik}=\frac{\d V^0_{i,\s(k)}}{q_{\s(k)}V_{0}}-\frac{\d V^{0}_{i,\s(0)}}{q_{\s(0)}V_{0}}\ .
\end{equation*}
Therefore (\ref{perm_equivarianza}) reduces to prove the following Lemma \ref{lm:equivarianza}.
\end{proof}

\begin{lemma}\label{lm:equivarianza} Consider $V=(\v_0,\ldots,\v_n)\in\mathfrak{V}_n$ and $\s\in\mathfrak{S}_{n+1}$. If $Q=(q_0,\ldots,q_n)$ and $\s(Q)=\left(q_{\s(0)},\ldots,q_{\s(n)}\right)$ then
\begin{equation}\label{perm_equiv_componenti}
    \forall\ 1\leq i\leq n\ ,\ \forall\ 1\leq k\leq n,\quad q_0 V^{\s(0)}_{i,\s(k)}=(-1)^{\s(0)}\left(q_{\s(0)}V^0_{i,\s(k)}-q_{\s(k)}V^0_{i,\s(0)}\right)\ ,
\end{equation}
assuming $V^0_{i,\s(k)}=0$ when $\s(k)=0$ and $V^0_{i,\s(0)}=0$ when $\s(0)=0$.
\end{lemma}

Observe that equations (\ref{perm_equiv_componenti}) are \emph{equivalent} to the right equivariance (44).

\begin{proof} Let us first of all observe that if $\s(0)=0$ then (\ref{perm_equiv_componenti}) is an identity.  Then we can assume $\s(0)\neq 0$ giving $\s$ as a product of transpositions of type $(0,h)$. Then we can reduce to prove (\ref{perm_equiv_componenti}) just for a transposition $\s=(0,h)$. Actually we can further assume $h=1$ by applying the transposition $(h,1)$, for which the equivariance (\ref{perm_equivarianza}) has been already proven. Notice that if $\s(k)=0$ then (\ref{perm_equiv_componenti}) reduces to the following
\begin{equation*}
    \forall\ 1\leq i\leq n\quad V^{\s(0)}_{i,0}=(-1)^{\s(0)-1}V^0_{i,\s(0)}\ ,
\end{equation*}
which can be easily verified. Then we can assume $\s(k)\neq 0$ which is $2\leq k\leq n$. A further simplifying step is transforming the matrix $V$ to the associated $Q$--canonical fan matrix, by reducing $V^0$ in HNF. Namely there exists a unique matrix $H^0\in\Mat(n,\Z)$ in HNF and such that $V^0=U\cdot H^0$, for some $U\in\GL(n,\Z)$, by Theorem \ref{thm:cohen}; then $H:=U^{-1}\cdot V$ is the $Q$--canonical fan matrix of $\P(Q)$, by Proposition \ref{prop:fan minimale}(4). Then
\begin{equation*}
    \tau(V\cdot A(\s))=\tau(U\cdot H\cdot A(\s))= U^*\cdot\tau(H\cdot A(\s))
\end{equation*}
by the already proven left equivariance of $\tau$. Assume that the statement holds for the $Q$-canonical matrix $H$, then we are able to write
\begin{equation*}
    \tau(V\cdot A(\s))=U^*\cdot\tau(H\cdot A(\s))=U^*\cdot\tau(H)\ast \s= \tau(V)\ast\s\ ,
\end{equation*}
ending up the proof. Therefore we have simply to prove that
\begin{equation}\label{perm_equiv_ridotta}
    \forall\ 1\leq i\leq n\ ,\ \forall\ 2\leq k\leq n,\quad q_0 H^{1}_{i,k}+q_{1}H^0_{i,k}-q_{k}H^0_{i,1}=0\ .
\end{equation}
Let us first of all consider the case $i=1$. Recall that
\begin{eqnarray*}
    H^0&=&
    \left(
                    \begin{array}{ccccc}
                      k_2 & v_{1,2} & \cdots& \cdots & v_{1,n} \\
                      0 & k_3/k_2 & \cdots &\cdots & v_{2,n} \\
                      \vdots & 0 & \ddots &  & \vdots \\
                      \vdots & \vdots&  & k_n/k_{n-1}& v_{n-1,n}\\
                      0 & 0 & \cdots & 0 & q_0/k_n \\
                    \end{array}
                  \right)\\
    H^1&=&\left(
                    \begin{array}{ccccc}
                      v_{1,0} & v_{1,2} & \cdots& \cdots & v_{1,n} \\
                      v_{2,0} & k_3/k_2 & \cdots &\cdots & v_{2,n} \\
                      \vdots & 0 & \ddots &  & \vdots \\
                      \vdots & \vdots& & k_n/k_{n-1}& v_{n-1,n}\\
                      v_{n,0} & 0 & \cdots&  0 & q_0/k_n \\
                    \end{array}
                  \right)\ .
\end{eqnarray*}
By Proposition \ref{prop:induzione} we get
\begin{equation*}
    H^1_{1,k}= \frac{(-1)^{1+k}}{k_2}\ \widehat{H}_k = \frac{q_k}{k_2}\ ,
\end{equation*}
while clearly
\begin{equation*}
    H^0_{1,k}=0\quad,\quad H^{0}_{1,1}= \frac{q_0}{k_2}\ .
\end{equation*}
Therefore
\begin{equation}\label{perm_equiv_componenti_1}
    \forall\ 2\leq k\leq n,\quad q_0 H^{1}_{1,k}+q_{1}H^0_{1,k}-q_{k}H^0_{1,1}=\frac{q_0q_k}{k_2}-\frac{q_kq_0}{k_2}=0\ .
\end{equation}
Finally, for $2\leq i\leq n$, observe that equations (\ref{perm_equiv_componenti_1}) are linear, hence invariant by left multiplication of elements of $\GL(n,\Z)$. Since the first and the $i$--th row can be exchanged each other by left multiplication of a suitable matrix $B(1,i)\in\mathfrak{S}_n\subset\GL(n,\Z)$, equations (\ref{perm_equiv_ridotta}) follows immediately by equations (\ref{perm_equiv_componenti_1}).
\end{proof}

Passing to the quotient by the right action of $\mathfrak{S}_{n+1}$ and recalling (\ref{Fn}), the right equivariance of the weighted transverse map $\tau$ gives rise to the following commutative diagram
\begin{equation}\label{diagramma}
    \xymatrix{\mathfrak{V}_n\ar[r]^{\tau}\ar@{>>}[d]^{\varphi}&\mathfrak{W}_n\ar@{>>}[d]^{\psi}\\
              \mathfrak{F}_n\ar[r]^-{\D}&\mathfrak{W}_n/\mathfrak{S}_{n+1}}
\end{equation}
where $\varphi$ and $\psi$ are the obvious quotient maps.

Let us say a few words about the induced map $\D$. Choose a particular fan $\Si\in\mathfrak{F}(Q)\subset\mathfrak{F}_n$ and recall maps $\D^j_Q$ defined in (\ref{fan-polytope}). The quotient map $\psi$ can be factorized by the action of the subgroup $\mathfrak{S}^0=\{\s\in\mathfrak{S}_{n+1}\ |\ \s(0)=0\}$, giving
\begin{equation*}
    \xymatrix{\mathfrak{W}_n\ar@{>>}[rr]^-{\psi}\ar@{>>}[dr]^-{\psi'}&& \mathfrak{W}_n/\mathfrak{S}_{n+1}\\
                &\mathfrak{W}_n/\mathfrak{S}^0\ar@{>>}[ur]^-{\psi''}&}
\end{equation*}
Images $\im(\D^j_Q),\ j=0,\ldots,n$,  can be embedded as subsets of the quotient $\mathfrak{W}_n/\mathfrak{S}^0$
by observing that an element $w\in\mathfrak{W}_n/\mathfrak{S}^0$ is given by the \emph{set} of columns of any $P$--admissible matrix in $\psi^{-1}(w)$ and setting
\begin{equation*}
    \conv(\mathbf{0},\w_1,\ldots,\w_n)\in\im(\D^j_Q)\longmapsto w=\{\w_1,\ldots,\w_n\}\in\mathfrak{W}_n/\mathfrak{S}^0\ .
\end{equation*}
For any $0\leq j\leq n$, we get then the following commutative diagram
\begin{equation*}
    \xymatrix{&\mathfrak{V}_n\ar[rr]^{\tau}\ar@{>>}[d]^{\varphi}&&\mathfrak{W}_n\ar@{>>}[d]^{\psi}\\
              &\mathfrak{F}_n\ar[rr]^-{\D}&&\mathfrak{W}_n/\mathfrak{S}_{n+1}\\
              \mathfrak{F}(Q)\ar@{^{(}->}[ur]\ar[rr]^-{\D^j_Q}&&**[r]\im(\D^j_Q)\subset\mathfrak{W}_n/
               \mathfrak{S}^0\subset\mathfrak{P}_n\ar@{^{(}->}[ur]_>>>>>{\psi''}&}
\end{equation*}
where $\psi''_{|\D^j_Q}$ turns out to be injective since its inverse corresponds to choose one of the $n+1$ possible representatives politopes of the class $\D(\Si)\in\mathfrak{W}_n/\mathfrak{S}_{n+1}$.

\begin{proposition}\label{prop:classification_Poly} In the previous diagram (\ref{diagramma}) the restricted maps $\tau_{|\mathfrak{V}_n^{\text{red}}}$ and $\D_{|\mathfrak{F}_n^{\text{red}}}$ are injective. Then, recalling Proposition \ref{prop:classification}, the following sets are bijectively equivalent
\begin{eqnarray}
    \label{1:1 bis}
    \left\{\text{wps's}\right\}/\text{iso}&\stackrel{1:1}{\longleftrightarrow}&\left\{Q\in(\N\setminus\{0\})^{n+1}\ |\ Q\ \text{is reduced}\right\}/\mathfrak{S}_{n+1}\\
    \nonumber
    &\stackrel{1:1}{\longleftrightarrow}&\GL(n,\Z)\left\backslash\mathfrak{F}_n^{\text{red}}\right.\ \stackrel{1:1}{\longleftrightarrow}\ \GL(n,\Z)\left\backslash\mathfrak{V}_n^{\text{red}}\right/\mathfrak{S}_{n+1}\\
    \nonumber
    &\stackrel{1:1}{\longleftrightarrow}& \GL(n,\Z)\left\backslash\mathfrak{W}_n\right/\mathfrak{S}_{n+1}\ .
\end{eqnarray}
\end{proposition}

\begin{proof} Both $\tau_{|\mathfrak{V}_n^{\text{red}}}$ and $\D_{|\mathfrak{F}_n^{\text{red}}}$ are injective due to the fibres structure of $\tau$, by Proposition \ref{prop:suriettività}, and passing to the quotient. Since those maps are also surjective, (\ref{1:1 bis}) follows. \end{proof}

\section{Computing cohomology}

This final section does not pretend to introduce any original content: the cohomology of a weighted projective space and of a line bundle on it, is quite well known! Anyway, by completeness, we'd like to recall some useful facts to reconstruct a (combina)-toric proof of formulas obtained by I.~Dolgachev in \cite{Dolgachev} \S 2.3.

\subsection{Homology and cohomology with rational coefficients}

For completeness let us recall the following well known result.

\begin{theorem}[\cite{Fulton} \S 4.5 and \S 5.2, \cite{Danilov} Prop. 12.11, \cite{Oda} Thm. 3.11]\label{thm:Q-coomologia} Let $X=X(\Si)$ be a $n$-dimensional toric variety associated with a \emph{simplicial} and \emph{complete} fan $\Si$ and set
$h_i:=\dim_{\Q}\left(H_{i}(X,\Q)\right)$.
Then
\begin{eqnarray*}
    h_{2k} &=& \sum_{i=k}^n (-1)^{i-k} {i\choose k} d_{n-i}\\
    h_{2k+1} &=& 0
\end{eqnarray*}
where $d_j$ is the number of $j$--dimensional cones in $\Si$. Moreover Poincar\'{e} Duality holds giving that $H^k(X,\Q)$ is dual to $H_k(X,\Q)$ and $h_i=h_{2n-i}$ for all $0\leq i\leq n$.
\end{theorem}

In the case $X$ is a wps we are now in a position to compute all the homology and cohomology with rational coefficients.

\begin{corollary}\label{cor:coomologia} Let $\P(Q)$ be a $n$-dimensional wps. Then
\begin{equation*}
    \forall\ 0\leq k\leq n\quad h_{2k}(\P(Q))= 1\quad,\quad h_{2k+1}(\P(Q))=0\ .
\end{equation*}
\end{corollary}

\begin{proof} If $X=\P(Q)$ then $d_{n-i}={n+1\choose n-i}$ and Theorem \ref{thm:Q-coomologia} gives
\begin{equation}\label{h_2k}
    h_{2k} = \sum_{i=k}^n (-1)^{i-k} {i\choose k}{n+1\choose n-i} \ .
\end{equation}
Notice that for $k=n$ (\ref{h_2k}) implies that $h_{2n}=1$. Then the proof ends up by observing that for all $1\leq k\leq n$
\begin{eqnarray*}
  h_{2(k-1)} - h_{2k}&=& {n+1 \choose n+1-k} +\sum_{i=k}^n (-1)^{i-k}{i+1\choose k}{n+1\choose n-i} \\
   &=&  \frac{(n+1)!}{k!(n+1-k)!}\left(1+\sum_{i=k}^n(-1)^{i-k}\frac{(n+1-k)!}{(i-k+1)!(n-i)!}\right)\\
   &=& \frac{(n+1)!}{k!(n+1-k)!}\left(1+\sum_{j=1}^{n+1-k}(-1)^j\frac{(n+1-k)!}{(j)!(n+1-k-j))!}\right)\\
   &=& \frac{(n+1)!}{k!(n+1-k)!}\left(1-1\right)^{n+1-k} = 0\ .
\end{eqnarray*}
\end{proof}

\subsection{Serre--Grothendieck duality}\label{ssez:serre} Recall that a $n$--dimensional toric variety $X$ is Cohen--Macauley,
 meaning that it admits a dualizing sheaf $\omega_X$ (see \cite{Oda} \S 3.2, \cite{Hartshorne} Thm.III.7.6). In particular if
 $i:U\hookrightarrow X$ is the open embedding of the nonsingular locus $U$ of $X$ and consider the Zariski sheaf of germs of $p$--forms $\Omega^p_X:=i_*\Omega^p_U$ then it is a coherent $\mathcal{O}_X$--module giving, for $p=n$, the dualizing sheaf i.e. $\omega_X\cong\Omega^n_X$ (\cite{Oda} Cor. 3.9).

\begin{theorem}[Serre--Grothendieck's Duality Theorem (\cite{Oda} \S 3.3)] Let $X=X(\Si)$ be a $n$--dimensional toric variety associated with a complete and simplicial fan $\Si$. Then the exterior product induces a perfect bilinear pairing
\begin{equation*}
    \forall\ p,q\in\N : p+q=n\quad H^q(X,\Omega^p_X)\otimes_{\C}H^{n-q}(X,\Omega^{n-p}_X)\longrightarrow H^n(X,\Omega^n_X)\cong \C\ ,
\end{equation*}
so that $H^{n-q}(X,\Omega^{n-p}_X)$ is canonically the dual $\C$--vector space of $H^q(X,\Omega^p_X)$ i.e.
\begin{equation*}
    H^{n-q}(X,\Omega^{n-p}_X)\cong H^q(X,\Omega^p_X)^{\vee}
\end{equation*}
Moreover if $\mathcal{F}$ is a locally free $\mathcal{O}_X$--module then
\begin{equation*}
    H^{n-q}(X,\mathcal{F}^{\vee}\otimes_{\mathcal{O}_X}\Omega^n_X)\cong H^q(X,\mathcal{F})^{\vee} \ .
\end{equation*}
\end{theorem}

When $X$ is a \emph{Gorenstein space} then the dualizing sheaf $\omega_X$ is an invertible sheaf. For an integral scheme this means that there exists a Cartier divisor $K_X$ on $X$ such that $\omega_X\cong\mathcal{O}_X(K_X)$ (\cite{Hartshorne} Prop.II.6.5). This is called \emph{the canonical divisor of $X$}. If $X=X(\Si)$ is a simplicial, complete and Gorenstein toric variety it turns out that
$$\Omega^n_X\cong\mathcal{O}_X\left(-\sum_{\rho\in\Si(1)}D_{\rho}\right)\quad\text{where}\quad K_X\equiv -\sum_{\rho\in\Si(1)}D_{\rho}\quad\text{is a Cartier divisor,} $$
with the same notation introduced in \ref{sssez:divisori} (see \cite{Oda} Cor. 3.3 and the following Remark). Let us finally recall that $X(\Si)$ is a \emph{Fano} toric variety if $-K_X=\sum_{\rho\in\Si(1)}D_{\rho}$ is an ample divisor. Then by Proposition \ref{prop:molta-ampiezza} we get the following

\begin{proposition}\label{prop:gorenstein} Let $Q=(q_0,\ldots,q_n)$ be a \emph{reduced} weights vector. Then the following facts are equivalent:
\begin{enumerate}
  \item $\P(Q)$ is Gorenstein,
  \item $\P(Q)$ is Fano,
  \item $\forall\ 0\leq j\leq n\quad q_j\ |\ |Q|:=\sum_{j=0}^n q_j$, which is\,: $\d\ |\ |Q|$.
\end{enumerate}
\end{proposition}

\begin{proof} By Theorem \ref{thm:P(Q)-divisori},(2) for all $0\leq j\leq n$ the divisor $\d/q_j\ D_j$ is a generator of the Picard group $\Pic(\P(Q))$. Then
$$K_{\P(Q)} \equiv -\frac{|Q|}{\d}\left(\frac{\d}{q_j}\ D_j\right)$$
is Cartier if and only if $\d\ |\ |Q|$, meaning that $\omega_{\P(Q)}\cong\mathcal{O}\left(-|Q|/\d\right)$. This is also equivalent to guarantee that $-K_{\P(Q)}$ is a (very) ample divisor, by Proposition \ref{prop:molta-ampiezza}.
\end{proof}

\subsection{Cohomology of line bundles} Recall the following well known result

\begin{theorem}[\cite{Danilov} 6.3 and Cor. 7.3, \cite{Oda} Thm 2.7, Cor. 2.8 and 2.9, \cite{Fulton} \S 3.5]\label{thm:lb-coomolgia} Let $X(\Si)$ be a complete toric variety and $D$ a Cartier divisor of $X$ such that $\mathcal{O}_X(D)$ is generated by its global sections. Then
\begin{equation*}
    h^q(X,\mathcal{O}_X(D))=\left\{\begin{array}{ccc}
                                     |M\cap \D_D| & \text{for} & q=0 \\
                                     0 & \text{for} & q\neq 0
                                   \end{array}
\right.
\end{equation*}
where $|M\cap \D_D|$ denotes the number of lattice points in the convex polytope $\D_D$.
\end{theorem}

If $X$ is the wps $\P(Q)$ then a line bundle $\mathcal{L}$ is generated by its global sections if and only $\mathcal{L}\cong\mathcal{O}(m)$ with $m\geq 0$, by Proposition \ref{prop:molta-ampiezza}. Let $\D$ be the polytope of a generator of $\Pic(\P(Q))$, e.g. consider $\D=\D_0:=\D_{\d/q_0 D_0}$. Then the previous result and the Serre--Grothendieck duality give immediately the following

\begin{proposition}\label{prop:lb-coomologia} For any $m\geq 0$
\begin{equation*}
    h^q(\P(Q),\mathcal{O}(m))=\left\{\begin{array}{ccc}
                                     |M\cap m\D| & \text{for} & q=0 \\
                                     0 & \text{for} & q\neq 0 \ .
                                   \end{array}
\right.
\end{equation*}
Moreover if $\P(Q)$ is Gorenstein then, for any $m\geq |Q|/\d$,
\begin{equation*}
    h^q(\P(Q),\mathcal{O}(-m))=\left\{\begin{array}{ccc}
                                     \left|M\cap \left(m-\frac{|Q|}{\d}\right)\D\right| & \text{for} & q=n \\
                                     0 & \text{for} & q\neq n \ .
                                   \end{array}
\right.
\end{equation*}
\end{proposition}

The previous Proposition does not say anything for line bundles $\mathcal{O}(-m)$ with $0<m<|Q|/\d$. This gap will be filled up in the next subsection (see Remark \ref{rem:punti}).

\subsection{Bott--Tu formulas for $\P(Q)$} The previous Theorem \ref{thm:lb-coomolgia} admits an extension to sheaves $\Omega^p_X\otimes_{\mathcal{O}_X}\mathcal{L}$ up to strengthen the hypothesis on the fan $\Si$ and the line bundle $\mathcal{L}$. Namely the following results hold

\begin{theorem}[\cite{Danilov} \S 12, \cite{Oda} Thm.3.11]\label{thm:hodge} Let $X(\Si)$ be a complete and simplicial $n$--dimensional toric variety. Then
\begin{equation*}
    h^q(X,\Omega_X^p)=\left\{\begin{array}{ccc}
                                     \sum_{i=p}^n (-1)^{i-p} {i\choose p} d_{n-i} & \text{for} & q=p \\
                                     0 & \text{for} & q\neq p
                                   \end{array}
\right.
\end{equation*}
where $d_j$ is the number of $j$--dimensional cones in $\Si$.
\end{theorem}

\begin{theorem}[\cite{Danilov} (7.5.2) and Thm.7.5.2, \cite{Oda} \S 3.3,  \cite{Fujino2} Cor.1.3]\label{thm:ample_bott-tu}  Let $X(\Si)$ be a complete toric variety and $D$ an ample divisor of $X$. Then, for all $0\leq p\leq n$,
\begin{equation*}
    h^q(X,\Omega_X^p\otimes_{\mathcal{O}_X}\mathcal{O}_X(D))=\left\{\begin{array}{ccc}
                                     \sum_{\u\in M\cap\D_D} {s(\u,D)\choose p} & \text{for} & q=0 \\
                                     0 & \text{for} & q\neq 0
                                   \end{array}
\right.
\end{equation*}
where $s(\u,D)$ denotes the dimension of the smallest face of $\D_D$ containing $\u$ and assuming ${s \choose p}=0$ for $s < p$ and ${0\choose 0}=1$.
\end{theorem}

Finally we are in a position to state and prove the following Bott--Tu formulas for wps, giving a convex geometric approach to Dolgachev's results in \cite{Dolgachev} \S 2.3. The advantage of this approach is that one reduces the cohomology computation to an enumeration of lattice points in suitable (faces of) polytopes, which is directly implementable.

\begin{theorem}[Bott-Tu formulas for $\P(Q)$]\label{thm:bott-tu}
Let $Q=(q_0,\ldots,q_n)$ be a reduced weights vector and define
$$h^q\Omega^p(m):=\dim_{\C} H^q\left(\P(Q),\Omega^p_{\P(Q)}\otimes_{\mathcal{O}_{\P(Q)}}\mathcal{O}(m)\right)\ .$$
Let $\D$ be the polytope associated with a generator of $\Pic(\P(Q))$. Then
\begin{enumerate}
  \item
\begin{equation*}
    h^0\Omega^p(m)=\left\{\begin{array}{ccc}
                            \sum_{\u\in M\cap m\D} {s(\u,m)\choose p} & \text{for} & m\geq 0\\
                            0 & \text{for} & m<0
                          \end{array}
\right.
\end{equation*}
where $s(\u,m)$ is the dimension of the smallest face in $m\D$ containing $\u$;
  \item for all $0<q<n$
\begin{equation*}
    h^q\Omega^p(m)=0\quad\text{for}\quad m\neq 0\ ;
\end{equation*}
moreover, for $m=0$,
\begin{equation*}
    h^q\Omega^p=\left\{\begin{array}{ccc}
                            1 & \text{for} & q= p\\
                            0 & \text{for} & q\neq p\ ;
                          \end{array}
\right.
\end{equation*}
  \item
\begin{equation*}
    h^n\Omega^p(m)=\left\{\begin{array}{ccc}
                            0 & \text{for} & m>0\\
                            \sum_{\u\in M\cap (-m)\D} {s(\u,-m)\choose n-p} & \text{for} & m\leq 0\ ;
                          \end{array}
\right.
\end{equation*}
in particular, for $m=0$, this gives
\begin{equation*}
    h^n\Omega^p=\left\{\begin{array}{ccc}
                            0 & \text{for} & p< n\\
                            1 & \text{for} & p=n\ .
                          \end{array}
\right.
\end{equation*}
\end{enumerate}
\end{theorem}

\begin{remark} Our notation $h^q\Omega^p(m)$ does not exactly coincide with the Dolgachev's notation $h(j,i;n)$ in \cite{Dolgachev} 2.3.2. Precisely the switching rule is the following
$$h^q\Omega^p(m)=h(q,p;m\d)$$
where $Q$ is reduced and $\d=\lcm(Q)$.
\end{remark}

\begin{proof}[Proof of Theorem \ref{thm:bott-tu}] (1): The first line is obtained applying Theorem \ref{thm:ample_bott-tu} for $m>0$ and Theorem \ref{thm:hodge} for $m=0$. The second line is still a consequence of Theorem \ref{thm:ample_bott-tu} after applying the Serre--Grothendieck duality \ref{ssez:serre}. In fact
\begin{equation*}
    \forall\ m<0\quad h^0\Omega^p(m)=h^n\Omega^{n-p}(-m)=0\ .
\end{equation*}

(2): The first equation is still obtained by Theorem \ref{thm:ample_bott-tu}, in case by applying Serre--Grothendieck duality when $m<0$. Formulas for $m=0$ are directly obtained by Theorem \ref{thm:hodge}, recalling that $d_{n-i}={n+1 \choose n-i}$ and proceeding as in the proof of Corollary \ref{cor:coomologia}.

(3): The first line comes directly from Theorem \ref{thm:ample_bott-tu}. For the second one, apply Serre--Grothendieck duality to get:
\begin{equation*}
    \forall\ m\leq 0\quad h^n\Omega^p(m)=h^0\Omega^{n-p}(-m)= \sum_{\u\in M\cap (-m)\D} {s(\u,-m)\choose n-p}
\end{equation*}
by Theorem \ref{thm:ample_bott-tu}.
\end{proof}

\begin{remark}\label{rem:punti} Notice that for $p=0$ Theorem \ref{thm:bott-tu} gives an improvement of results in Proposition \ref{prop:lb-coomologia}, without any gap on $m$ and any further hypothesis on $\P(Q)$.

\noindent In particular, giving the further assumption that $\P(Q)$ is Gorenstein, which is that $\d\ |\ |Q|$ by Proposition \ref{prop:gorenstein}, we get
\begin{equation}\label{punti}
    \forall\ m\geq \frac{|Q|}{\d}\quad \left|M\cap \left(m-\frac{|Q|}{\d}\right)\D\right|=\left|M\cap (m\D)^{\circ}\right|
\end{equation}
where $\D$ is the polytope associated with a generator of $\Pic(\P(Q))$ and $(m\D)^{\circ}$ denotes the interior of the polytope $m\D$. Namely this means that
\begin{itemize}
  \item[(a)] \emph{the number of lattice points internal to the polytope $m\D$ coincides with the number of lattice points belonging to the polytope $\left(m-\frac{|Q|}{\d}\right)\D$}.
\end{itemize}
 Moreover if $-|Q|/\d<m<0$ then Theorem \ref{thm:bott-tu}(3) gives $h^n\mathcal{O}(m)=|M\cap (-m\D)^{\circ}|$. On the other hand Corollary 2.3.3 in \cite{Dolgachev} proves that $h(n,0;\d m)=0$ if $-\d m<|Q|$, which is precisely that $h^n\mathcal{O}(m)=0$ in our range. This proves that
\begin{itemize}
  \item[(b)] \emph{if $0\leq m< |Q|/\d$ then the rational polytope $m\D$ does not admits internal lattice points,}
\end{itemize}
the case $m=0$ being obvious.
\end{remark}

\end{document}